\documentclass[11pt,reqno,a4paper]{amsart}

\usepackage{amsfonts}
\usepackage{epsfig}
\usepackage{graphicx}
\usepackage{amsmath}
\usepackage{amssymb}
\usepackage{color}
\usepackage{mathrsfs}
\usepackage{txfonts}
\usepackage{mathrsfs}
\usepackage[colorlinks=true,linkcolor=blue, urlcolor=red, citecolor=blue]{hyperref}	
\usepackage{subfigure}

\makeatletter \@addtoreset{equation}{section} \makeatother

\renewcommand\thetable{\thesection.\@arabic\c@table}

\theoremstyle{plain}
\newtheorem{maintheorem}{Theorem}
\newtheorem{theorem}{Theorem }[section]
\newtheorem{proposition}[theorem]{Proposition}

\newtheorem{lemma}[theorem]{Lemma}

\newtheorem{maincorollary}{Corollary}
\newtheorem{claim}{Claim}

\theoremstyle{definition} \theoremstyle{remark}
\newtheorem{remark}[theorem]{Remark}

\newtheorem{definition}[theorem]{Definition}

\newcommand{\al} {\alpha}

\newcommand{\vep}{\varepsilon}

\newcommand{\supp}{\operatorname{supp}}
\newcommand{\diam}{\operatorname{diam}}
\newcommand{\dist}{\operatorname{dist}}

\newcommand{\Leb}{Leb}

\newcommand{\cU}{\mathcal{U}}

\theoremstyle{remark}

\theoremstyle{definition}

\theoremstyle{plain}

\makeatletter
\@addtoreset{equation}{section}

\begin{document}

\title[Ergodic optimization for hyperbolic flows and Lorenz attractors]
    {Ergodic optimization for hyperbolic flows and Lorenz attractors}

\author{Marcus Morro, Roberto Sant'Anna and Paulo Varandas}

\address{Marcus Morro, Departamento de Matem\'atica e Estat\'istica, Universidade Federal da Bahia\\
Av. Ademar de Barros s/n, 40170-110 Salvador, Brazil}
\email{marcusmorro@gmail.com}

\address{Roberto Sant'Anna, Departamento de Matem\'atica e Estat\'istica, Universidade Federal da Bahia\\
Av. Ademar de Barros s/n, 40170-110 Salvador, Brazil}
\email{rsantanna@ufba.br}

\address{Paulo Varandas, Departamento de Matem\'atica e Estat\'istica, Universidade Federal da Bahia\\
Av. Ademar de Barros s/n, 40170-110 Salvador, Brazil $\&$ CMUP, University of Porto, Rua do Campo Alegre 687, 4169--007 Porto, Portugal}
\email{pcvarand@gmail.com}
\urladdr{https://sites.google.com/view/paulovarandas/}


\maketitle
 \begin{abstract}
In this article we consider the ergodic optimization for hyperbolic flows and Lorenz attractors
with respect to both continuous and H\"older continuous observables. In the context of hyperbolic flows
we prove that a Baire generic subset of continuous observables
have a unique maximizing measure, with full support and zero entropy, and that a Baire generic subset of
H\"older continuous observables admit a unique and periodic maximizing measure.
These results rely on a relation between ergodic optimization for suspension semiflows
and ergodic optimization for the Poincar\'e map with respect to induced observables, which allow
us to reduce the problem for the context of maps. Using that singular-hyperbolic attractors are
approximated by hyperbolic sets, we obtain related results for geometric Lorenz attractors.
 \end{abstract}

\section{Introduction and statement of the main results}

Let $X$ be a compact metric space, $f:X\rightarrow X$ be a continuous map and $\mathcal{M}_{f}$ be the collection of $f$-invariant
Borel probability measures on $X$. The objects of interest in the field of ergodic optimization are those $f$-invariant probability measures which maximize, or minimize, the space average $\int \varphi d\mu$,
for $\varphi:X\to\mathbb{R}$, over all $\mu\in\mathcal{M}_{f}$. These are the maximizing measures, or minimizing, measures for the function $\varphi$ (with respect to the dynamical system $f$). As usual, we restrict our attention to maximizing measures, since a minimizing measure for $\varphi$ is a maximizing measure for $-\varphi$. The compactness of $\mathcal M_f$ and continuity of the
function $\mu\mapsto \int \varphi d\mu$ ensures that maximizing measures always exist. It is also clear from the ergodic decomposition
theorem that almost all ergodic components of a maximizing measure are maximizing measures, hence ergodic maximizing measures
also exist.
Hence, some of the fundamental question in ergodic optimization are:
\begin{enumerate}
\item[$\circ$] What can we say about the maximizing measures?
\item[$\circ$] Is there only one maximizing measure for typical observables?
\item[$\circ$] Can we describe the support of a maximizing measure?
\item[$\circ$] Are maximizing measures typically periodic?
\end{enumerate}
There exists an extensive list of contributions to these problems built over different approaches, some of which inspired by
statistical mechanics and thermodynamic formalism (zero temperature limits) and others from the theory of cohomology
equations (construction of sub-actions).
In the known situations, the answer to the previous questions usually depend on the
class of the dynamics but also on the regularity of the observables.
In~\cite{mane1997lagrangian} Ma\~n\'e  conjectured that for a generic Lagrangian
 there exists a unique minimizing measure, and it is supported by a periodic orbit.
Contreras, Lopes and Thieullen
\cite{contreras2001lyapunov} and later Contreras \cite{contreras2016ground} obtained a proof of this conjecture in the case
of expanding maps.  For an account on the many contributions to this problem we refer the reader to \cite{BLL, BiGa, BYiwei, bousch2000poisson,Bremont,fathi1997kam,garibaldi2009calibrated, Yiwei1,Yiwei2, morris2007, morris2010ergodic,quas2012super,yuan1999optimal}
and references therein.
Based on various approaches utilized in the literature, we can emphasize that the regularity of the observables plays an important role on the proofs: for Lipschitz potentials, one can obtain maximizing measures supported in periodic orbits, whereas for continuous potentials, the support of the maximizing measure is the whole space. We refer the reader to \cite{baraviera2013ergodic, jenkinson2006ergodic, jenkinson2018} for excellent surveys on ergodic optimization.

Here we will address on the ergodic optimization for hyperbolic and singular-hyperbolic flows with respect to both continuous and H\"older continuous observables.
Let $M$ be a closed Riemannian manifold and $(X^{t})_{t}\colon M\to M$
a smooth flow. Given a continuous function $\varphi:M\to\mathbb{R}$
a \emph{maximizing measure} for $(X^{t})_{t}$ with respect to $\varphi$ is a $(X^{t})_{t}$-invariant
Borel probability measure $\mu$ so that
\[
\int\varphi d\mu=\max\Big\{ \int\varphi \,d\nu\colon \nu\in\mathcal{M}_{1}(M,(X^{t})_{t})\Big\} .
\]
Maximizing measures always exist because $\mathcal{M}_{1}(M,(X^{t})_{t})$
is compact in the weak{*} topology and $\nu\mapsto\int\varphi d\mu$
is continuous.
First results on the ergodic optimization for flows were due to Lopes and Thieullen \cite{lopes2005sub} and Pollicott and Sharp~\cite{PS2} where the authors constructed sub-actions for Anosov flows (related results include \cite{LRR} in the context of
expansive geodesic flows).
The construction of calibrated sub-actions (that is, normalized by the maximal average) can be understood as a first step
in the direction of ergodic optimization as these can be used to identify the support of maximizing measures.
A second breakthrough was obtained by Contreras~\cite{contreras2014} in the context of Lagrangian dynamics,
which proves that $C^2$-generic hyperbolic Ma\~n\'e sets contain a periodic orbit and that it actually reduces to a single periodic
orbit in the case of surfaces.

The main goal here is to contribute to the ergodic optimization of singular-hyperbolic attractors in three-dimensional manifolds, where
the geometric Lorenz attractors form the paradigmatic examples. Although geometric Lorenz attractors admit a global cross-section, one cannot tackle this problem directly and to reduce their ergodic optimization to the ergodic optimization of their Poincar\'e maps.
Indeed,
the presence of singularities makes not only the roof function to be piecewise smooth and unbounded, as the
Poincar\'e return map is generally non-Markovian and just piecewise smooth with unbounded derivatives (cf.~\cite{araujo2010three}).
While one could expect the ideas in \cite{lopes2005sub} to be useful to construct calibrated sub-actions in the previous context
for a suitable (countable) Markov inducing scheme,  a complete ergodic optimization description seems
far from unattainable by this approach.

Our strategy to overcome the previous difficulties exploit the fact that singular-hyperbolic attractors can be
approximated by horseshoes (see \cite{araujo2010three,Young}). Indeed, a singular-hyperbolic set with
no singularities on a three-dimensional manifold is uniformly hyperbolic (see e.g. \cite{araujo2010three}).
On the one hand, the existence of Markov partitions for hyperbolic flows
\cite{Bowen, ratner1973markov} ensures that the suspended horseshoes can be modeled by
suspension flows over a subshift of finite type.
Then we prove that the ergodic optimization of hyperbolic flows can be reduced to the ergodic optimization of bilateral
subshifts of finite type (and later to one-sided subshifts of finite type) with respect to induced observables. We prove that
the previous reduction has a  fibered structure in the space of observables, namely that is formed by submersions in the space
of observables, and use the latter to
prove that results on the ergodic optimization for bilateral subshifts of finite type lead to a translation of such results for
suspended horseshoes (see Section~\ref{sec:reducao}).  In particular, adapting \cite{contreras2016ground,morris2010ergodic,Shinoda} to the context of topologically mixing bilateral subshifts of
finite type we prove that for each of these approximating suspended horseshoes: (i)
there exists a open and dense set of H\"older observable with a unique maximizing measure, supported on a periodic orbit;
(ii) there is a $G_\delta$-dense set of continuous observables with a unique maximizing measure, and it has zero entropy and
full support; and (iii) there exists a $C^0$-dense subset formed by observables that admit uncountable many ergodic
maximizing measures with positive entropy.
Related results for the singular-hyperbolic attractors are obtained by an approximation argument,
explored in Section~\ref{secLorenz}.

\medskip

Our main results can be grouped according to both the regularity and hyperbolicity of the
flow, and the regularity of the observables.

\subsubsection*{Hyperbolic flows}

Assume that $(X^t)_t$ is a $C^1$-flow and that $\Lambda$ is a hyperbolic basic set that is conjugated to a suspension flow
over a subshift of finite type (see ~Subsection~\ref{s:hyp} for the definitions). We prove that typical continuous observables
have unique and zero entropy maximizing measures. More precisely:

\begin{maintheorem} \label{TeoA:hyper-subset} Let $M$ be a $d$-dimensional
compact boundaryless Riemannian manifold and $(X^{t})_{t\in\mathbb{R}}$
be a $C^{1}$-flow in $M$. If $\Lambda\subset M$ is
a hyperbolic basic set for $(X^{t})_{t\in\mathbb{R}}$ that is conjugated to a suspension flow
over a subshift of finite type then the following hold:
\begin{enumerate}
\item there exists an open and dense
set $\mathcal{O}\subset C^{\alpha}(M,\mathbb{R})$ of observables
$\varphi:M\to\mathbb{R}$ such that, for every $\varphi\in\mathcal{O}$,
there is a unique $(X^{t})_t$-maximizing measure and it is supported on a periodic orbit;
\item there exists a dense $G_{\delta}$ set
$\mathcal Z\subset C^{0}(M,\mathbb{R})$ such that for every $\varphi\in \mathcal Z$,
there is a single $(X^{t})_{t\in\mathbb{R}}$-maximizing measure,
it has zero entropy and support equal to $\Lambda$; and
\item there exists a dense set
$\mathcal D \subset C^{0}(M,\mathbb{R})$ such that for every $\varphi\in \mathcal D$,
there exists uncountably many $(X^{t})_{t\in\mathbb{R}}$-invariant and ergodic
maximizing measures.
\end{enumerate}
\end{maintheorem}

Since hyperbolic flows admit Markov partitions, these may be modeled by suspension flows.
We observe that Theorem~\ref{TeoA:hyper-subset} will follow from a more general result
on suspension flows (cf. Theorem~\ref{thm:semi-flow}).

\subsubsection*{Lorenz attractors}
Our next results concern wild Lorenz attractors
(we refer the reader to Subsection~\ref{s:shyp} for the definition).
We prove the following.

\begin{maintheorem}\label{TeoALorenz} Let $M$ be a 3-dimensional compact
boundaryless Riemannian manifold and $\Lambda$ be a wild Lorenz attractor
for a flow $(X^{t})_{t}:M\to M$. Then:
\begin{enumerate}
\item there exists a $C^{0}$-residual subset $\mathcal R_1\subset C^{0}(M,\mathbb{R})$
such that for every $\varphi\in \mathcal R_1$ there is an unique $(X^{t})_{t}$-maximizing
measure $\mu$ with respect to $\varphi$; moreover,
$\mu$ is not atomic and the support $\supp \mu$ contains the singularity;
\item there is a $C^{\al}$-residual subset $\mathcal R_2\subset C^{\alpha}(M,\mathbb{R})$
of $\alpha$-H\"older observables such that, for every $\varphi\in \mathfrak R_2$
there is an unique $(X^{t})_{t}$-maximizing measure $\mu$;
moreover, either $\mu$ is supported on a critical element (a singularity or a periodic orbit)
or it is non-atomic whose support contains some singularity.
\end{enumerate}
\end{maintheorem}

One expects the previous result to hold for singular-hyperbolic attractors in general
(see Subsection~\ref{s:hyp} for the definition). However, the argument in proof of Theorem~\ref{TeoALorenz}
explores transitivity of locally maximal subsets and a characterization of the space of invariant measures
for the wild Lorenz attractors (see e.g. Lemma~\ref{le:mper}).
More precisely, the conclusion of Theorem~\ref{TeoALorenz} holds for all Lorenz attractors so that the set of
periodic measures is dense in the convex space of invariant probabilities, a condition which holds
for wild Lorenz attractors.

\medskip

One other comment concerns invariant measures whose support contains the singularity at Theorem~\ref{TeoALorenz}.
In general we cannot ensure that these measures are full supported on the attractor.
Nevertheless, in the case of $C^1$-generic vector fields, we prove that every three-dimensional singular-hyperbolic attractor
(including the Lorenz attractor) coincides with
the closure of the unstable manifold of its singularities (cf. Proposition~\ref{MP}).
Then, if we endow the space $\mathfrak{X}^1(M) \times C^0(M,\mathbb R)$ with the product topology we have
the following consequence:

\begin{maincorollary}\label{cor:generic}
Let $M$ be a 3-dimensional compact boundaryless Riemannian manifold.
There exists a Baire residual subset $\mathcal R \subset \mathfrak{X}^1(M)\times C^0(M,\mathbb R)$
such that for any pair $(X,\varphi) \in \mathcal R$ there exists a unique $(X^{t})_{t}$-maximizing
measure $\mu$ with respect to $\varphi$ on each Lorenz attractor
$\Lambda$ for the flow $(X^t)_t$ generated by $X$.
Moreover, $\supp \mu=\Lambda$.
\end{maincorollary}

Some comments are in order. We could not rule out the possibility of having maximizing measures
that are not supported at critical elements in item (2) of Theorem~\ref{TeoALorenz}.
Furthermore, we observe that the condition that the support contains the singularity implies on a non-trivial
recurrence to the singularity which can be thought as a replacement in this context to the fact that these measures
are expected to have large support.

While this work was being written it was brought to our attention the results by Huang et al \cite{Yiwei2}, which
establish ergodic optimization for Axiom A flows using very different methods and establish a continuous-time
Ma\~n\'e-Conze-Guivarch-Bousch lemma.

\medskip
This paper is organized as follows. In Section~\ref{prelim} we give some preliminaries on ergodic optimization in the
discrete time setting, suspension flows and weak forms of hyperbolicity for flows. In Section~\ref{sec:bilateral} we prove
some results on the ergodic optimization for bilateral subshifts of finite type. The method explores a functional analytic
description of the reduction using the solutions of the cohomological equation. A method for recovering results on the ergodic
optimization for suspension semiflows from their counterpart for the Poincar\'e map is developed along Section~\ref{sec:reducao},
where we also prove Theorem~\ref{TeoA:hyper-subset}. In Section~\ref{secLorenz} we use that
Lorenz attractors
are approximated by horseshoes in order to characterize the space of invariant probabilities for Lorenz attractors and to
prove Theorem~\ref{TeoALorenz}. Finally, some final comments are addressed in Section~\ref{sec:final}.


\section{Preliminaries}\label{prelim}

\subsection{Ergodic optimization for maps }\label{LE}

In this subsection we recall some contributions for the ergodic optimization of maps.
If $N$ is a compact metric space, $f:N\to N$ is a continuous map
and $\psi:N\to\mathbb{R}$ is continuous, a \emph{maximizing measure} for $f$ with respect to $\psi$
is an $f$-invariant Borel probability measure $\bar{\mu}$ which maximizes the integral of $\psi$ among all $f$-invariant
Borel probabilities. In other words,
\[
\int\psi~ d\bar{\mu}=\max\Big\{ \int\psi ~d\bar{\nu}\colon\bar{\nu}\in\mathcal{M}_{1}(N,f)\Big\} .
\]
We denote $M(\psi,f)=\max\Big\{ \int\psi \,d\bar{\nu}\colon \bar{\nu}\in\mathcal{M}_{1}(N,f)\Big\} $.
As discussed in the introduction there is a dichotomy depending on the regularity of the
observables and structure of the underlying dynamics. The first results consider continuous maps
with Bowen's specification property (see e.g. \cite{BTV} for definition and a further discussion).

\begin{theorem} \cite[Theorem 3.2]{jenkinson2006ergodic}\label{Jenkinson}
Let $f: N \to N$ be a continuous map on a compact metric space $N$, and let
$E$ be a topological vector space which is densely and
continuously embedded in $C^0(N,\mathbb R)$. Then the set of observables
$\varphi\in E$ that have a unique maximizing measure is a countable intersection of open and dense subsets of $E$.
In particular, if $E$ is a Baire space then the set above is dense in $E$.
\end{theorem}

\begin{remark}\label{rmk:Holder}
Given $\alpha>0$, the previous theorem ensures that there exist Baire residual subsets in $C^0(N,\mathbb R)$ and in
$C^\alpha(N,\mathbb R)$ formed by observables with a unique maximizing measure.
\end{remark}

\begin{remark}\label{rmk:thm-unique-flows}
Theorem~\ref{Jenkinson} admits a counterpart to continuous flows. Indeed, the argument in its proof
relies on convergences both on the space of measures (in the weak$^*$ topology) and the space of
(continuous) potentials, and it does not depend on the discrete-time or continuous-time nature of the dynamics itself. Thus,
the set of observables having a unique maximizing measure (invariant by a continuous flow) forms a Baire generic subset on both the spaces of continuous and H\"older continuous observables.
\end{remark}

By the previous discussion, for any fixed continuous flow $(X^t)_t$, typical observables have a unique maximizing measure.
We would like to say more about these measures (e.g. to characterize the support of these measures). Before recalling such kind of
results for maps we need a definition.

\begin{definition}\label{def:gluing1}
We say that $f$ satisfies the \emph{gluing orbit property} if for any $\vep > 0$ there exists an integer $m = m(\vep) \geq 1$ so that for any $x_1^{}, x_2^{}, \dots , x_k^{} \in N$ and any integers $n_1^{}, \dots , n_k^{} \ge 0$ there are $0 \leq p_1^{}, \dots , p_{k-1}^{} \leq m(\vep)$ and a
point $y \in N$ so that $\displaystyle d(f^j (y), f^j (x_1^{}))\leq \vep$ for every $0 \leq j \leq n_1^{}$ and
$
d(f^{j+n_1^{}+p_1^{}+ \dots +n_{i-1}^{}+p_{i-1}^{}} (y), f^j (x_i^{}))\leq \vep
$
for every $2 \leq i \leq k$ and $0 \leq j \leq n_i^{}$.
If, in addition, $y\in N$ can be chosen periodic with period $\sum_{i=1}^k (n_i+p_i)$
for some $0 \leq p_{k}^{} \leq m(\vep)$ then we say that $f$ satisfies the \emph{periodic gluing orbit property}.

\end{definition}

The latter is a condition weaker than Bowen's specification and it is satisfied by transitive hyperbolic dynamics,
and minimal equicontinuous maps, among other class of examples (cf. ~\cite{BV,BTV} and references therein).

\begin{theorem} \label{morris}
Let $f:N\to N$ be a continuous transformation of a compact metric space satisfying the periodic gluing orbit property.
Then there is a dense $G_{\delta}$ set $Z\subset C^{0}(N,\mathbb{R})$ such that every $\varphi \in Z$ has a unique $\varphi$-maximizing measure, it has full support in $N$
and zero entropy.
\end{theorem}

\begin{proof}
This result is a simple modification of \cite[Corollary 1.3]{morris2010ergodic}. Indeed, Corollary~1.2 in \cite{morris2010ergodic}
ensures that the ergodic maximizing measures of a generic continuous function have the same properties as
generic ergodic measures. Hence, it suffices to check that the gluing orbit property can replace Bowen's specification
as a mechanism to prove that full supported and zero entropy measures are Baire generic.

Since the proof follows Sigmund's approach in \cite{Sig70} closely, with minor modifications, we just emphasize the diferences in the argument.
Let $\{U_i\}_{i\ge 0}$ be a countable basis of the topology in $N$.
For any $i\ge 0$, by weak$^*$ convergence it is clear that the space $\mathcal E_i$ of invariant probabilities $\mu$
so that $\mu(U_i)=0$ is closed.  We claim that this set has also empty interior. Indeed, let $\mu\in \mathcal M_1(f)$ be such that
$\mu(U_i)=0$, and let $\vep>0$ and $\varphi_i\in C^0(N,\mathbb R)$ be arbitrary and define the open neighborhood
$$\mathcal V= \mathcal V(\mu; \varphi_1, \dots, \varphi_k, \frac\vep4)
		:=\Big\{\nu \in \mathcal M_1(f) \colon \big|  \int \varphi_i \, d\mu - \int \varphi_i \, d\nu\big| < \frac\vep4, \; \forall 1\le i \le k\Big\}
$$
of $\mu$, in the weak$^*$-topology. By uniform continuity, there exists $\delta>0$ so that $|\varphi_i(x)-\varphi_i(y)|<\frac\vep4$
 whenever $d(x,y)<\delta$.  Assume further, reducing $\delta$ is necessary, that there exists $x_0\in U_i$ so that $B(x_0,2\delta)\subset U_i$.

 By Birkhoff's ergodic theorem there exists a $\mu$-full measure subset $N_0\subset N$ so that for every $x\in N_0$
and $1\le i \le k$ the limit
$
\varphi_i^*(x)= \lim_{n\to\infty} \frac1n\sum_{j=0}^{n-1} \varphi_i(f^j(x))
$
exists and $\int \varphi_i^*\, d\mu=\int \varphi_i \,d\mu$.
Let $\mathcal P$ be a finite partition of $N_0$ so that $\sup_{x\in P} \varphi_i^*(x) - \inf_{x\in P} \varphi_i^*(x) < \frac\vep4$
for every $P\in \mathcal P$ and every $1\le i \le k$. In consequence, for   $x_P\in P$ arbitrary,
$$
\big| \int \varphi_i \, d\mu - \sum_{P\in \mathcal P} \mu(P) \, \varphi_i^*(x_P)\big|
	= \big| \int \varphi_i^* \, d\mu - \sum_{P\in \mathcal P} \mu(P) \, \varphi_i^*(x_P)\big| <\frac\vep4
$$
Choose $N_1\ge 1$ so that $|\frac1n \sum_{j=0}^{n-1} \varphi_i(f^j(x_P)) - \int \varphi_i \, d\mu |< \frac\vep4$
for every $1\le i \le k$, every $P\in \mathcal P$ and every $n\ge N_1$. Choose $N_2 \ge N_1$ so that
every $m \ge N_2$ can be written as
$$
m=\sum_{P\in\mathcal P} m_P
	\quad\text{where}\quad
	\big| \frac{m_P}{m} -\mu(P) \big| \le \frac{\vep}{12 \, [\# \mathcal P] \,\max_{1\le i \le k} \|\varphi_i\|_0}.
$$
Finally, by the gluing orbit property, for each $N_3\ge N_2$ there exists a periodic point $z\in N$ so that $d(z,x_0)<\delta$
and for each $P \in \mathcal P$ the orbit of $z$ shadows the finite piece of orbit
$\{x_{P}, f(x_{P}), \dots, f^{N_3-1}(x_{P}) \}$ for $m_P$ consecutive times, recursively, with time lags in between the
orbits of length at most $m(\delta)$  (here $m=m(\delta)$ is given by the gluing orbit property at scale $\delta$).
The period of $z$ is
$$
\pi_z:=1+ p_0 + \sum_{j=1}^{\# \mathcal P} (N_3 + p_j)
$$
where $0\le p_j \le m(\delta)$ for every $0\le j \le \#\mathcal P$.
It is clear from the construction that $\mu_z:=\frac1{\pi_z} \sum_{j=0}^{\pi_z-1} \delta_{f^j(z)}$ is such that $\mu_z(U_i)>0$.
Moreover, one can choose $N_3\gg N_2$ large enough so that the proportion of the orbit of the point $z$
outside of the shadowing process satisfies
$$
\frac{\sum_{j=0}^{\# \mathcal P}  p_j}{1+ p_0 + \sum_{j=1}^{\# \mathcal P} (N_3 + p_j)}
	\le \frac{ m(\delta) \, \# \mathcal P }{1 + N_3 \, \# \mathcal P }
	< \frac{\vep}{12  \,\max_{1\le i \le k} \|\varphi_i\|_0}.
$$
In particular, the proof that $\mu_z\in \mathcal V$ will follow the same lines of \cite[pp. 104]{Sig70}.
Altogether,
the latter shows that each $\mathcal E_i$ is a closed set with empty interior, hence the invariant
measures giving positive measure to open sets form a Baire generic subset.
The proof that zero entropy measures form a Baire generic subset is identical and left as an exercise to the reader.
\end{proof}

We now focus on expanding dynamics.
Given a compact metric space  $N$,  a map $T:N\to N$ is \textit{Ruelle expanding} if there are $k\in\mathbb{Z}^{+}$ and $0<\lambda<1$ such that for every point $x\in N$ there is
a neighborhood $U_{x}$ of $x$ in $N$ and continuous branches $S_{i}$, $i=1,\ldots,\ell_{x}\le k$ of the inverse of $T$ such that
$
T^{-1}(U_{x})=\bigcup_{i=1}^{\ell_{x}}S_{i}(U_{x}),\ \ T\circ S_{i}=I_{U_{x}}
$
for all $i$, and
$
d\left(S_{i}(y),S_{i}(z)\right)\le\lambda d(y,z)
$
for all $y,z\in U_{x}$.
Transitive Ruelle expanding maps satisfy the periodic gluing orbit property~\cite{BV}.
Throughout, assume without loss of generality that $\text{diam}\,N=1$ and
let $C^{\alpha}(N,\mathbb{R})$ denote the space of $\alpha$-H\"{o}lder
observables (ie. $\varphi\in C^{\alpha}(N,\mathbb{R})$ if there are constants
$C,\alpha>0$
so that
$
|\varphi(x)-\varphi(y)|\le Cd(x,y)^{\alpha}
$
for all $x,y\in N$).

\begin{theorem} \cite{contreras2016ground} \label{thmContreras} If $N$ is a compact metric space and $f:N\to N$
is a Ruelle expanding map there is an open and dense set $\mathcal{O}\subset C^{\alpha}(N,\mathbb{R})$
so that every $\varphi\in\mathcal{O}$ admits a unique $\varphi$-maximizing
measure and it is supported on a periodic orbit.
\end{theorem}

\begin{remark}
As stated above, Theorem~\ref{thmContreras} differs from the version presented in \cite{contreras2016ground},
which was stated for the space $\text{Lip}(N,\mathbb{R})$ of Lipschitz observables instead of $C^{\alpha}(N,\mathbb{R})$.
Nevertheless, Theorem~\ref{thmContreras} is a direct consequence of the main result in \cite{contreras2016ground}
together with the fact that, since $\diam N=1$, the space of $\alpha$-H\"older continuous functions
with respect to the metric $d(\cdot,\cdot)$ coincides with the space of Lipschitz functions
with respect to the modified metric
$d_{\alpha}(\cdot,\cdot)=d(\cdot,\cdot)^{\alpha}$.
\end{remark}

Finally we recall Shinoda's result on the dense non-uniqueness of maximizing measures.

\begin{theorem}\cite{Shinoda}\label{Shinoda}
Let $(\Sigma_R,\sigma)$ be a one-sided topologically mixing subshift of finite type. There exists a dense subset $\mathcal D \subset
C^0(\Sigma_R,\mathbb R)$ such that for every $\varphi \in \mathcal D$ there exist uncountably many ergodic $\varphi$-maximizing measures with full support and positive entropy.
\end{theorem}

\subsection{Hyperbolic flows}\label{s:hyp}

Let $M$ be a closed Riemannian manifold and let $(X^{t})_{t}\colon M\to M$
be a smooth flow. Let $\Lambda\subseteq M$ be a compact
$(X^{t})_{t}$-invariant set. The flow $(X^{t})_{t}\colon\Lambda\to\Lambda$
is \emph{uniformly hyperbolic} if for every $x\in\Lambda$ there exists
a $DX^{t}$-invariant and continuous splitting $T_{x}M=E_{x}^{s}\oplus E_{x}^{X}\oplus E_{x}^{u}$
and constants $C>0$ and $0<\theta_{1}<1$ such that
\begin{equation}
\|DX^{t}\mid E_{x}^{s}\|\leq C\theta_{1}^{t}\quad\text{and}\quad\|(DX^{t})^{-1}\mid E_{x}^{u}\|\leq C\theta_{1}^{t},\label{eqhip}
\end{equation}
for every $t\ge0$. We say that $(X^{t})_{t}$ is an Anosov flow if
$(X^{t})_{t}\colon M\to M$ is uniformly hyperbolic.
It is well known that adapted
metrics exist, hence we may assume $C=1$.
Given a hyperbolic set $\Lambda$,
for each $x\in\Lambda$ consider the stable and the unstable manifolds
$
W^{s}(x)=\{y\in M:\text{dist}(X^{t}(y),X^{t}(x)) \to 0 \;\text{as}\; t\to+\infty\}
$
and
$
W^{u}(x)=\{y\in M:\text{dist}(X^{t}(y),X^{t}(x)) \to 0 \;\text{as}\; t\to-\infty\},
$
respectively.
By the stable manifold theorem, uniform hyperbolicity ensures that there exists $\vep>0$ so that
the largest connected components $W_{loc}^{s}(x)\subset W^{s}(x)$ and $W_{loc}^{u}(x)\subset W^{u}(x)$
of size $\vep$ containing $x$ are smooth submanifolds, called respectively (local) stable and unstable manifolds
(of size $\vep$) at the point $x$. Moreover:
\begin{enumerate}
\item $T_{x}W_{loc}^{s}(x)=E^{s}(x)$ and $T_{x}W_{loc}^{u}(x)=E^{u}(x)$;
\item for each $t>0$ we have
$
X^{t}(W_{loc}^{s}(x))\subset W_{loc}^{s}(x)(X^{t}(x))\ \ \text{and}\ \ X^{-t}(W_{loc}^{u}(x))\subset W_{loc}^{u}(X^{-t}(x));
$
\item there exist $\kappa>0$ and $\gamma\in(0,1)$ such that for each $t>0$
we have
$
d(X^{t}(y),X^{t}(x))\le\kappa\gamma^{t}d(y,x)\ \ \text{for}\ \ y\in W_{loc}^{s}(x),
$
and $d(X^{-t}(y),X^{-t}(x))\le\kappa\gamma^{t}d(y,x)\ \ \text{for}\ \ y\in W_{loc}^{u}(x).$
\end{enumerate}
Moreover, the set $\Lambda$ is \emph{locally maximal}
if there exists an open neighborhood $U$ of $\Lambda$ such that
$
\Lambda=\bigcap_{t\in\mathbb{R}}X^{t}(U).
$

A point $p \in M$ is a \emph{singularity} for $X$ if $X(p)=0$,
and is called a \emph{regular point} otherwise.
We say that a singularity $p$ is \emph{hyperbolic} if the one-point invariant set $\{ p \}$ is a hyperbolic set.
A point $p\in M$ is \emph{periodic} if there exists a minimum period
$T>0$ so that $X^T(p)=p$, and we say that $p$ is a \emph{periodic hyperbolic point} if the orbit
$\mathcal O(p)=\cup_{t\in [0,T]} X^t(p)$ is a hyperbolic set for $X$.
Finally, (an orbit of) a point $x$ by the flow is called a \emph{critical element} if it is either periodic or $x$ is a singularity.

Now let $\Lambda$ be a locally maximal hyperbolic set. For any sufficiently
small $\vep$, there exists a $\delta>0$ such that if $x,y\in\Lambda$
are at a distance $d(x,y)\le\delta$, then there exists a unique $t=t(x,y)\in[-\vep,\vep]$
for which the set
$
[x,y]=W_{loc}^{s}(x)(X^{t}(x))\cap W_{loc}^{u}(y)
$
is a single point in $\Lambda$ (see e.g. \cite[Proposition 7.2]{hirsch1970neighborhoods}).

\begin{definition}
\label{def:basic_set}We say that $\Lambda$ is a \emph{hyperbolic
basic} set if
(i) $\Lambda$ is hyperbolic (not a fixed point);
(ii) the periodic orbits of $(X^{t})_{t}|\Lambda$ are dense in $\Lambda$;
(iii) $(X^{t})_{t}|\Lambda$ is transitive (contains a dense orbit);
(iv) $\Lambda$ is locally maximal.
We say $(X^{t})_{t}$ is \emph{Axiom A} if $\Omega$ is the
disjoint union of hyperbolic sets and a finite number of hyperbolic
fixed points. Moreover, we say that $\Lambda$ is a \emph{horseshoe} if
it is topologically conjugated to the suspension flow over a subshift of finite type.
\end{definition}

\subsection{Singular-hyperbolic and Lorenz attractors}\label{s:shyp}

We say that a compact $(X^t)_{t\in\mathbb R}$-invariant set $\Lambda\subset M$ is \emph{partially hyperbolic} if there are a continuous
invariant splitting $T_{\Lambda}M=E^s\oplus E^c$, constants $C>0$ and $\lambda\in (0,1)$ so that
\begin{equation*}
\|D_xX^t|E_x^s\|\le C\lambda^t \quad \text{ and } \quad \|D_xX^t|E_x^s\| \cdot \|D_{X^t(x)}X_{-t}|E^c_{X^t(x)}\|\le C\lambda^t
\end{equation*}
for every $x\in\Lambda$ and $t\ge 0$.
If, in addition, the following two conditions (i) and (ii) hold, then we say that $\Lambda$ is \emph{sectional-hyperbolic}:
\begin{itemize}
\item[(i)] every singularity $p\in\Lambda$ is hyperbolic;
\item[(ii)] $E^c$ is sectionally expanding, i.e. $\dim E^c\ge 2$ and $|\det(D_xX^t\mid_{L_x})|\ge C^{-1}\lambda^t$ for every $x\in\Lambda$, $t\ge 0$, and every two-dimensional subspace $L_x\subset E^c_x$.
\end{itemize}
With some abuse of notation, we say that the flow $(X^t)_{t\in \mathbb R}$ is partially hyperbolic if $M$ is a partially hyperbolic set.
\textcolor{black}{$\Lambda$ is said to be \emph{singular-hyperbolic} if $\Lambda$ is partially hyperbolic such that it satisfies the above condition (i) and $E^c$ is volume expanding, i.e. $|\det(D_xX^t\mid_{E^c_x})|\ge C^{-1}\lambda^t$ for every $x\in\Lambda$, $t\ge 0$.}

Finally we give a brief description of the construction of {geometric Lorenz attractors}. We will follow  \cite{araujo2010three}
(see also \cite[Section 3]{ArVar}).
Let $\Sigma=\{(x,y,1)\in\mathbb{R}^{3}:|x|,|y|\le1\}$ and $\Gamma=\{(0,y,1)\in\mathbb{R}^{3}:|y|\le1\}$.
Consider a $C^1$-vector field $X_0$ on $\mathbb{R}^{3}$ so that the following conditions hold:
\begin{enumerate}
\item For any point $(x,y,z)$ in a neighborhood of the origin $(0,0,0)$
of $\mathbb{R}^{3}$, $X_0$ is given by $$(\dot{x},\dot{y},\dot{z})=(\lambda_{1}x,-\lambda_{2}y,-\lambda_{3}z)$$
where $0<\lambda_{3}<\lambda_{1}<\lambda_{2}$.
\item All forward orbits of $X$ starting from $\Sigma\backslash\Gamma$
will return to $\Sigma$ and the first return map $P:\Sigma\backslash\Gamma\to\Sigma$
is a piecewise $C^{1}$-diffeomorphism which has the form
\begin{equation}\label{eq:defPoincare}
P(x,y,1)=(\alpha(x),\beta(x,y),1),
\end{equation}
where $\alpha:[-1,1]\backslash\{0\}\to[-1,1]$ is a piecewise $C^{1}$-map
with $\alpha(-x)=-\alpha(x)$ and there is $0<\gamma<1$ satisfying
\begin{equation}\label{eq:defLorenz}
\begin{cases}
\alpha(x) =x^\gamma & \text{in a neighborhood of}\; 0 \\
\alpha'(x) >\sqrt 2, & \text{for any}\ x\text{\ensuremath{\neq}}0 \\
 \alpha(1)<1, & \\
\lim_{x\text{\ensuremath{\to}}0^{+}}\alpha(x)=-1, &\\
\lim_{x\text{\ensuremath{\to}}0^{+}}\alpha(x)=\text{\ensuremath{\infty}}, &
\end{cases}
\end{equation}
and there exists $\lambda\in (0,1)$ so that $|\frac{\partial \beta}{\partial y}(x,y)|\le \lambda$ for every $(x,y,1)\in \Sigma$.
\end{enumerate}
The second condition in \eqref{eq:defLorenz} ensures that the one-dimensional map $\alpha:[-1,1]\backslash\{0\}\to[-1,1]$ is
\emph{locally eventually onto}: for any open interval $I\subset [-1,1]\backslash\{0\}$ there exists $N\ge 1$ so that $\alpha^N(I)=(\alpha(-1),\alpha(1))$.
(cf. \cite[Lemma 3.16]{araujo2010three}).
We say that a one-dimensional Lorenz map is \emph{wild} if $\sup_\mu \int \log |f'| \, d\mu=+\infty$, where the supremum is taken
over all $f$-invariant probability measures $\mu$.

\begin{remark}\label{rmk:slow-rec}
Although the derivative of Lorenz maps is unbounded it can occur that for a particular Lorenz map $\alpha$
all orbits have slow recurrence to the singular point, causing all invariant measures to have uniformly bounded
Lyapunov exponent (e.g. this is the case when the singular point is pre-periodic repelling). While it is expected for both classes
of wild and non-wild geometric Lorenz attractors to be locally dense (in a $C^2$-neighborhood of the original vector field),
it was recently announced that non-wild Lorenz maps are actually generic along special parameterized families of Lorenz attractors \cite{Pedreira}.
\end{remark}

There exists an open elipsoid $V \subset \mathbb R^3$ containing the origin such that the vector field $X_0$ points
inwards, hence it exhibits an attractor.
If $\mathcal U \subset \mathfrak{X}^1(\mathbb R^3)$ is a $C^1$-open set of the vector field $X_0$ and an open elipsoid
$V \subset \mathbb R^3$ containing the origin such that every $X\in \cU$ exhibits a partially hyperbolic attractor
$\Lambda_X = \bigcap_{t\ge 0} \overline{X^t(V)}$,
and it is called \emph{geometric Lorenz attractor}.
We say that $\Lambda$ is a \emph{wild Lorenz attractor} if the corresponding one-dimensional Lorenz map $\alpha$, obtained
by quotient along local stable leaves in the cross-section $\Sigma$, is wild.

It is well known that for every $X\in \mathcal U$ there exists a periodic point
$p_X \in \Lambda_X$ so that the Lorenz attractor $\Lambda_X$ coincides with the homoclinic class
$H(p_X):=\overline{W^{s}(p_X) \pitchfork W^{u}(p_X)}$ (cf. \cite[Proposition~3.17]{araujo2010three}).
In particular, the attractor
is transitive and, by Birkhoff-Smale's theorem (see e.g. \cite{KH}) it admits a
dense set of hyperbolic periodic orbits.
Moreover, any singular-hyperbolic attractor without singularities is uniformly hyperbolic
(see e.g. \cite{{araujo2010three}}).

As three-dimensional singular-hyperbolic attractors have only hyperbolic singularities whose unstable manifolds are one-dimensional
(see e.g. \cite{araujo2010three}) we denote by $W^{u,+}(\sigma)$ and $W^{u,-}(\sigma)$ the connected components of
$W^{u}(\sigma)\setminus \{\sigma\}$.
We finish this subsection with the following characterization for $C^1$-generic singular-hyperbolic attractors.

\begin{proposition}\label{MP}
There is a residual subset $\mathcal R\subset \mathfrak{X}^1(M)$ such that if $\Lambda$ is a singular-hyperbolic attractor for $X\in
\mathcal R$ and $\Lambda \cap \text{Sing}(X)\neq\emptyset$ then $\Lambda =\overline{W^{u,+}(\sigma)}=\overline{W^{u,-}(\sigma)}$
for every $\sigma\in \Lambda \cap \text{Sing}(X)$.
\end{proposition}

\begin{proof}
This results is a consequence of Theorem 4.2 in \cite{MP}. Indeed, the argument in \cite[pp 372--373]{MP}
uses the $C^1$-connecting lemma in order to prove that $\overline{W^{u}(\sigma)}$ is a Lyapunov stable set
for every singularity $\sigma\in \Lambda \cap \text{Sing}(X)$ of a $C^1$-generic vector field $X$.
Nevertheless, the argument follows without changes for the set $\overline{W^{u,+}(\sigma)}$ (and $\overline{W^{u,-}(\sigma)}$) instead of $\overline{W^{u}(\sigma)}$. Hence, there exists a $C^1$-residual subset $\mathcal R_*\subset \mathfrak{X}^1(M)$
so that if $\Lambda$ is a singular-hyperbolic attractor for $X\in \mathcal R_*$ and $\Lambda \cap \text{Sing}(X)\neq\emptyset$ then
$\Lambda =\overline{W^{u,*}(\sigma)}$ for every $\sigma\in \Lambda \cap \text{Sing}(X)$, for $*\in \{+,-\}$.
The $C^1$-residual subset $\mathcal R=\mathcal R_+\cap \mathcal R_-$ satisfies the requirements of the proposition.
\end{proof}

\begin{remark}\label{rmk:Xueting}
All non-atomic ergodic measures for singular-hyperbolic attractors can be approximated
by periodic measures due to the shadowing lemma in \cite{STV}. Indeed, by the presence of singularities the only difference
is that the closing lemma should be replaced by the extended Liao closing lemma in \cite{HW}.
We are grateful to X. Tian for pointing out this fact to us.
\end{remark}

\subsection{Suspension semiflows}\label{s:semiflows}

Let $f:N\rightarrow N$ be a continuous map on a compact metric space $(N,d_{N})$
and let $r:N\rightarrow(0,\infty)$ be a continuous function
bounded away from zero. Consider the quotient space
\begin{equation}
N^{r}=\Big\{ (x,t):0\le t\le r(x),x\in N\Big\} /\sim\label{eq:suspflow}
\end{equation}
where $(x,r(x))\sim(f(x),0)$.
The \emph{suspension semiflow over $f$ with height function $r$} is
the semiflow $(X^{t})_{t\in\mathbb{R}_+}$ in $N^{r}$ with $X^{t}:N^{r}\to N^{r}$
defined by $X^{t}(x,s)=(f^{n}(x),s^{\prime})$, where $n$ and $s^{\prime}$
are uniquely determined by
$
\sum_{i=0}^{n-1}r(f^{i}(x))+s^{\prime}=t+s\; \text{and}\; \ 0\le s^{\prime}<r(f^{n}(x)).
$
If $f$ is a homeomorphism then the previous expression defines a flow.
\medskip

We recall Bowen and Walters distance for suspension flows \cite{bowen1972expansive}.
Assume without loss of generality that the diameter of $N$ is bounded by one.
We first assume that the height function $r$ is constant equal to
1. Given $x,y\in N$ and $t\in[0,1]$, we define the length of
the horizontal segment $[(x,t),(y,t)]$ by
\[
\rho_{h}((x,t),(y,t))=(1-t)d_{N}(x,y)+td_{N}(f(x),f(y)).
\]
Clearly,
$
\rho_{h}((x,0),(y,0))=d_{N}(x,y)\ \ \text{e}\ \ \rho_{h}((x,1),(y,1))=d_{N}(f(x),f(y)).
$
 Moreover, given points $(x,t),(y,s)\in N^{r}$ in the same orbit,
we define the length of the vertical segment $[(x,t),(y,s)]$ by
\[
\rho_{v}((x,t),(y,s))=\inf\Big\{ |q|:X^{q}(x,t)=(y,s)\ \text{e}\ q\in\mathbb{R}\Big\} .
\]
For the height function $r=1$, the Bowen-Walters distance $d((x,t),(y,s))$
between two points $(x,t),(y,s)\in N^{r}$ is defined as the infimum
of the lengths of all paths between $(x,t)$ and $(y,s)$ that are
composed of finitely many horizontal and vertical segments.
Now we consider an arbitrary continuous height function $r:N\rightarrow(0,\infty)$
and we introduce the Bowen-Walters distance $d_{N^{r}}$ in $N^{r}$.
\begin{definition}
Given $(x,t),(y,s)\in N^{r}$ , we define
\[
d_{N^{r}}((x,t),(y,s))=d((x,t/r(x)),(y,s/r(y))),
\]
where $d$ is the Bowen-Walters distance for the height function $r=1$.
\end{definition}

For an arbitrary function $r$, a horizontal segment takes the form
$
w=[(x,t\cdot r(x)),(y,t\cdot r(y))],
$
and its length is given by
$
\ell_{h}(w)=(1-t)d_{N}(x,y)+td_{N}(f(x),f(y)).
$
 Moreover, the length of a vertical segment $w=[(x,t),(x,s)]$ is
now
$
\ell_{v}(w)=|t-s|/r(x),
$
for any sufficiently close $t$ and $s$.
It is sometimes convenient to measure distances in another manner.
Namely, given $(x,t),(y,s)\in N^{r}$, let
\begin{equation}
d_{\pi}((x,t),(y,s))=\min\left\{ \begin{array}{c}
d_{N}(x,y)+|t-s|,\\
d_{N^{r}}(f(x),y)+r(x)-t+s,\\
d_{N^{r}}(x,f(y))+r(y)-s+t,
\end{array}\right\} .\label{eqdpi(2.18)}
\end{equation}
Although $d_{\pi}$ may not be a distance it is related to the Bowen-Walters distance.

\begin{proposition}\cite[Proposition 2.1]{barreira2013dimension}
\label{prop:cdpi<dMr<cdpi(2.1)}If $f$ is an invertible Lipschitz
map with Lipschitz inverse, then there exists a constant $c\ge1$
such that
$
c^{-1}d_{\pi}(p,q)\le d_{N^{r}}(p,q)\le cd_{\pi}(p,q)
$
for every $p,q\in N^{r}.$
\end{proposition}

\begin{remark}\label{rem:measures}
Given a $(X^{t})_{t}$-invariant measure $\mu$ there exists an $f$-invariant probability $\tilde{\mu}$ on $N$ such that $\mu=\tilde \mu \times Leb /\int r\, d\tilde \mu$. It is well known that if $r$ is bounded away from zero then
$\tilde \mu \mapsto \tilde \mu \times Leb /\int r\, d\tilde \mu$
is a bijection between the space $\mathcal{M}_{1}(N,f)$ of $f$-invariant probabilities and the space $\mathcal{M}_{1}(N^{r},(X^{t})_{t})$
of $(X^{t})_{t}$-invariant probabilities.
\end{remark}

\section{Ergodic optimization for bilateral subshifts}\label{sec:bilateral}

\subsection{Symbolic dynamics}

Let $\Sigma_{n}=\{1,\ldots,n\}^{\mathbb{Z}}$ be the space of all
sequences $\underline{x}=\{x_{i}\}_{i=-\infty}^{\infty}$ with $x_{i}\in\{1,\ldots,n\}$
for all $i\in\mathbb{Z}$. We define the \emph{(left) shift} homeomorphism
$\sigma:\Sigma_{n}\to\Sigma_{n}$ by $\sigma\left(\{x_{i}\}_{i=-\infty}^{\infty}\right)=\{x_{i+1}\}_{i=-\infty}^{\infty}$.
If $\mathbf{R}$ is a $n\times n$ transition matrix formed by 0's and 1's, and
\[
\Sigma_{\mathbf{R}}=\left\{ \underline{x}\in\Sigma_{n}:\mathbf{R}_{x_{i}x_{i+1}}=1\ \text{for all}\ i\in\mathbb{Z}\right\} ,
\]
we say $\sigma:\Sigma_{\mathbf{R}}\to\Sigma_{\mathbf{R}}$
a \emph{subshift of finite type} (determined by $\mathbf{R}$). We denote by
$\Sigma_{n}^{+}=\{1,\ldots,n\}^{\mathbb{N}}$
the space of all sequences $\underline{x}=\{x_{i}\}_{i=0}^{\infty}$
with $x_{i}\in\{1,\ldots,n\}$ for all $i\in\mathbb{N}$ and define
the \emph{one-sided (left) shift} homeomorphism $\sigma:\Sigma_{n}^{+}\to\Sigma_{n}^{+}$
by $\sigma\left(\{x_{i}\}_{i=0}^{\infty}\right)=\{x_{i+1}\}_{i=0}^{\infty}$.
One-sided subshifts of finite type are defined analogously.
It is easy to check that a one-sided subshift of finite type is a Ruelle expanding map.

For $\varphi:\Sigma_{\mathbf{R}}\to\mathbb{R}$ continuous we define
the \emph{variation of $\varphi$ on $k$-cylinders} by
\[
\text{var}_{k}\varphi=\sup\{|\varphi(\underline{x})-\varphi(\underline{y})|:x_{i}=y_{i}\ \text{for all}\ |i|\le k\}.
\]
Let $\mathscr{F}_{\mathbf{R}}$ be the family of all continuous observables  $\varphi:\Sigma_{\mathbf{R}}\to\mathbb{R}$
for which there exists positive constants $b$ and $c\in(0,1)$ so that
$\text{var}_{k}\varphi\le bc^{k}$ for all $k\ge0$.

\begin{remark}
\label{rem:F_R_Holder} For any $\beta\in(0,1)$ one can define the
metric $d_{\beta}$ on $\Sigma_{\mathbf{R}}$ by $d_{\beta}(\underline{x},\underline{y})=\beta^{N}$
where $N$ is the largest non-negative integer with $x_{i}=y_{i}$
for every $|i|<N$. Then $\mathscr{F}_{\mathbf{R}}$ is the set of
functions which have a positive H\"{o}lder exponent with respect
to $d_{\beta}$. In fact, for $\underline{x},\underline{y}\in\Sigma_{\mathbf{R}}$
there is $N\in\mathbb{N}$ such that $d_{\beta}(x,y)=\beta^{N}$,
this means that $\underline{x}$ and $\underline{y}$ are in the same $N$-cylinder
and any $\varphi\in\mathscr{F}_{\mathbf{R}}$ satisfies
$
\text{var}_{N}\varphi  \le bc^{N},
$
which implies
$
|\varphi(\underline{x})-\varphi(\underline{y})|\le bc^{N}.
$
Choosing $\alpha\in(0,1)$ such that $c\le\beta^{\alpha}$ we have
$
|\varphi(\underline{x})-\varphi(\underline{y})|  \le
	b(\beta^{N})^{\alpha}
  =bd_{\beta}(\underline{x},\underline{y})^{\alpha}.
$
This means that $\varphi$ is $\alpha$-H\"{o}lder
in the metric $d_{\beta}$.
\end{remark}

\subsection{From unilateral to bilateral subshifts of finite type}\label{unilateral-bilateral}

Here we
build over the results for expanding maps in Subsection~\ref{LE} and the following classical result
(see e.g. \cite[Lemma 1.6]{bowen1975equilibrium})
on solutions of the cohomological equation, to
consider the ergodic optimization of bilateral subshifts of finite type.

\begin{lemma}
 \label{lemBowen1.6homology}If $\varphi\in\mathscr{F}_{\mathbf{R}}$,
then there exists a continuous function $u=u_\varphi:\Sigma_{\mathbf{R}}\to\mathbb{R}$
such that $\psi:=\varphi+u\circ\sigma-u\in\mathscr{F}_{\mathbf{R}}$
and $\psi(\underline{x})=\psi(\underline{y})$ whenever $x_{i}=y_{i}$
for all $i\ge0$.
\end{lemma}

\begin{proof}
Although this is well known lemma, we include its proof for the reader's convenience,
as we shall need the expression for the solution of the cohomological equation.
For each $1\le t\le n$ pick $\{a_{k,t}\}{}_{k=-\infty}^{\infty}\in \Sigma_{R}$
with $a_{0,t}=t$. Define $\rho:\Sigma_{R}\rightarrow\Sigma_{R}$
by $\rho(\underline{x})=\underline{x}^{*}$, where
$x_{k}^{*}=x_{k} \, \text{ for } \,k >0$ and $x_{k}^{*}= a_{k,x_{0}}  \,\text{for }\, k\le0$.
Let
\begin{equation}\label{eq:cobu}
u(\underline{x})=\sum_{j=0}^{\infty}(\varphi(\sigma^{j}(\underline{x}))-\varphi(\sigma^{j}(\rho(\underline{x})))).
\end{equation}
Note that
$
|\varphi(\sigma^{j}(\underline{x}))-\varphi(\sigma^{j}\rho(\underline{x})))|\le\text{var}_{j}\varphi\le b\alpha^{j}
$
because $\sigma^{j}(\underline{x})$ and $\sigma^{j}(\rho(\underline{x}))$
agree in places from $-j$ to $+\infty$.
As $\sum_{j=0}^{\infty}b\alpha^{j}<\infty$, the function $u$ is well defined
and continuous. If $x_{i}=y_{i}$ for all $|i|\le n$, then, for $j\in[0,n]$,
$
|\varphi(\sigma^{j}(\underline{x}))-\varphi(\sigma^{j}(\underline{y}))|\le\text{var}_{n-j}\varphi\le b\alpha^{n-j}
$
and
$
|\varphi(\sigma^{j}\rho((\underline{x})))-\varphi(\sigma^{j}\rho((\underline{y})))|\le b\alpha^{n-j}.
$
Hence
\begin{align*}
|u(\underline{x})-u(\underline{y})| & \le\sum_{j=0}^{\left[\frac{n}{2}\right]}|\varphi(\sigma^{j}(\underline{x}))-\varphi(\sigma^{j}(\underline{y}))+\varphi(\sigma^{j}\rho((\underline{x})))-\varphi(\sigma^{j}\rho((\underline{y})))|+2\sum_{j>\left[\frac{n}{2}\right]}\alpha^{j}\\
 & \le2b\left(\sum_{j=0}^{\left[\frac{n}{2}\right]}\alpha^{n-j}+\sum_{j>\left[\frac{n}{2}\right]}\alpha^{j}\right)
  \le\frac{4b\alpha{}^{\left[\frac{n}{2}\right]}}{1-\alpha}.
\end{align*}
This shows that $u\in\mathscr{F}_{\mathbf{R}}$. Hence the function $\psi:=\varphi+u\circ\sigma-u$
belongs to $\mathscr{F}_{\mathbf{R}}$ and it satisfies
\begin{align*}
\psi(\underline{x}) & =\varphi(\underline{x})+\sum_{j=-1}^{\infty}\left(\varphi(\sigma^{j+1}\rho((\underline{x})))-\varphi(\sigma^{j+1}(\underline{x}))\right)+\sum_{j=0}^{\infty}\left(\varphi(\sigma^{j+1}(\underline{x}))-\varphi(\sigma^{j}(\rho(\underline{x})))\right)\\
 & =\varphi(\sigma\rho((\underline{x})))+\sum_{j=0}^{\infty}\left(\varphi(\sigma^{j+1}(\underline{x}))-\varphi(\sigma^{j}(\rho(\underline{x})))\right).
\end{align*}
The final expression in the right-hand side above depends only on $\{x_{i}\}_{i=0}^{\infty}$,
as desired.
\end{proof}

Now we analyze the coboundary map \eqref{eq:cobu} as a function of the observable. By Remark \ref{rem:F_R_Holder},
up to a change of metric we have that the space $\mathscr{F}_{\mathbf{R}}$ coincides with the space of H\"older continuous
observables.
Hence we have the following:
\begin{lemma}
\label{lemHoloSubmersion} Let $D^{+}$ be the set of observables
$\psi\in C^{\alpha}(\Sigma_{\mathbf{R}},\mathbb{R})$ so that $\psi(\underline{x})=\psi(\underline{y})$ whenever $ x_{i}=y_{i}\ \text{for all}\ i\ge0$.
Then the map $\Xi:C^{\alpha}(\Sigma_{\mathbf{R}},\mathbb{R})\to D^{+}$
given by $\Xi(\varphi)=\varphi+u\circ\sigma-u$, where $u=u_{\varphi}:\Sigma_{\mathbf{R}}\to\mathbb{R}$
is given by Lemma \ref{lemBowen1.6homology}, is a submersion.
\end{lemma}

\begin{proof}
A simple computation shows that
the map $u:C^{\alpha}(\Sigma_{\mathbf{R}},\mathbb{R})\to C^{\alpha}(\Sigma_{\mathbf{R}},\mathbb{R})$
given by
$
u(\varphi)(\underline{x})=\sum_{j=0}^{\infty}(\varphi(\sigma^{j}(\underline{x}))-\varphi(\sigma^{j}(\rho(\underline{x}))))
$
is well defined and linear,
hence $\Xi:C^{\alpha}(\Sigma_{\mathbf{R}},\mathbb{R})\to D^{+}$ is also linear.
As $\Xi$ is surjective, by construction,
we conclude that $\Xi:C^{\alpha}(\Sigma_{\mathbf{R}},\mathbb{R})\to D^{+}$ is a submersion.
\end{proof}

\begin{remark}
\label{remShiftUnilateralBilateral}
If $\varphi\in C^{\alpha}(\Sigma_{\mathbf{R}},\mathbb{R})$
the observable $\tilde \varphi \in C^{\alpha}(\Sigma_{\mathbf{R}},\mathbb{R})$
defined by
$
\tilde{\varphi}(\{x_{i}\}_{i=-\infty}^{\infty})=\varphi(\{x_{i}\}_{i=0}^{\infty})
$
is constant
along local stable leaves. Reciprocally, if $\tilde{\varphi}\in C^{\alpha}(\Sigma_{\mathbf{R}},\mathbb{R})$
satisfies $\tilde{\varphi}(\underline{x})=\tilde{\varphi}(\underline{y})$
whenever $x_{i}=y_{i}$ for all $i\ge0$, then one can associate an
observable in $C^{\alpha}(\Sigma_{\mathbf{R}}^{+},\mathbb{R})$ by
$
\varphi(\{x_{i}\}_{i=0}^{\infty})=\tilde{\varphi}(\{x_{i}\}_{i=-\infty}^{\infty}).
$
The functions in $C^{\alpha}(\Sigma_{\mathbf{R}}^{+},\mathbb{R})$
are thus identified with the subclass of $C^{\alpha}(\Sigma_{\mathbf{R}},\mathbb{R})$
formed by functions that are constant on local stable leaves. Indeed, given the identification
$
\Sigma_{\mathbf{R}}^{+}\ \ \simeq \ \ \Sigma_{\mathbf{R}}/\sim,
$
where $\underline{x}\sim\underline{y}$ if $x_{i}=y_{i}$ for all
$i\ge0$ and $\underline{x},\underline{y}\in \Sigma_{\mathbf{R}}$,
one can identify
$
C^{\alpha}(\Sigma_{\mathbf{R}}^{+},\mathbb{R})\ \ \simeq\ \ C^{\alpha}(\Sigma_{\mathbf{R}},\mathbb{R})/\sim\ \ \simeq\ \ D^{+}.
$
\end{remark}

The  next proposition is the main result in this subsection, and it extends Theorem~\ref{thmContreras}
for bilateral subshifts of finite type.

\begin{proposition}
\label{prop:open_dense_shift_bilateral} There is an open and dense
subset of $\mathcal{R}\subset C^{\alpha}(\Sigma_{\mathbf{R}},\mathbb{R})$
such that every $\varphi\in\mathcal{R}$ admits a unique $\varphi$-maximizing
measure and it is supported on a periodic orbit of $\sigma:\Sigma_{\mathbf{R}}\to\Sigma_{\mathbf{R}}$.
\end{proposition}

\begin{proof}
Since a transitive one-sided subshift of finite type is a Ruelle expanding map we can apply Theorem \ref{thmContreras}
to $\sigma:\Sigma_{\mathbf{R}}^{+}\to\Sigma_{\mathbf{R}}^{+}$ and
obtain an open and dense set $\mathcal{O}\subset C^{\alpha}(\Sigma_{\mathbf{R}}^{+},\mathbb{R})$
such that for each $\varphi\in\mathcal{O}$ there is a single $\varphi$-maximizing
measure and it is supported on a periodic orbit. By the isomorphism in Remark \ref{remShiftUnilateralBilateral},
there exists
an open and dense set $\mathcal{O}^{+}\subset D^{+}$ such that every $\varphi\in\mathcal{O}^{+}$
has a single $\varphi$-maximizing measure and it is supported on
a periodic orbit. In fact, for every $\sigma$-invariant probability $\mu$
in $\Sigma_{\mathbf{R}}^{+}$ there is a natural way to associate a unique invariant probability $\tilde \mu$ on $\Sigma_{R}$.
Following \cite[Section C]{bowen1975equilibrium}, for $\varphi\in C^{0}(\Sigma_{\mathbf{R}}^{+},\mathbb{R})$
define $\varphi^{*}\in C^{0}(\Sigma_{\mathbf{R}}^{+},\mathbb{R})$ by
$$
\varphi^{*}\left(\{x_{i}\}_{i=0}^{\infty}\right):=\min\{\varphi(\underline{y}):\underline{y}\in \Sigma_{\mathbf{R}},y_{i}=x_{i}\ \text{for all}\ i\ge0\}.
$$
Notice that for $m,n\ge0$ one has
$
\|(\varphi\circ\sigma^{n})^{*}\circ\sigma^{m}-(\varphi\circ\sigma^{m+n})^{*}\|\le\text{var}_{n}\tilde{\varphi}.
$
Hence
\begin{align*}
\left|\int(\varphi\circ\sigma^{n})^{*}d\mu-\int(\varphi\circ\sigma^{n+m})^{*}d\mu\right| & =\left|\int(\varphi\circ\sigma^{n})^{*}\circ\sigma^{m}d\mu-\int(\varphi\circ\sigma^{n+m})^{*}d\mu\right|\\
 & \le\text{var}_{n}\tilde{\varphi}
\end{align*}
tends to zero as $n\to\infty$ (because $\varphi$ is continuous).
This proves that the latter is a Cauchy sequence and that the limit
$
\int\tilde{\varphi}d\tilde{\mu}=\lim_{n\to\infty}\int(\varphi\circ\sigma^{n})^{*}d\mu
$
exists. It is straightforward to check that
$\tilde{\mu}\in C^{0}(\Sigma_{\mathbf{R}},\mathbb{R})^{*}$. By the
Riesz Representation Theorem we see that $\tilde{\mu}$ defines a
probability measures on $\Sigma_{\mathbf{R}}$. Note that
$
\int\varphi\circ\sigma d\tilde{\mu}=\lim_{n\to\infty}\int(\varphi\circ\sigma^{n+1})^{*}d\mu=\int\tilde{\varphi}d\tilde{\mu}
$
for every continuous $\varphi$,
proving that $\tilde{\mu}$ is $\sigma$-invariant. Also $\int\tilde{\varphi}d\tilde{\mu}=\int\varphi d\mu$
for $\varphi\in C^{0}(\Sigma_{\mathbf{R}}^{+},\mathbb{R})$.

Note that if $\psi=\varphi+u\circ\sigma-u$, then $M(\varphi,\sigma)=M(\psi,\sigma)$
and the maximizing measures for $\varphi$ and $\psi$ are the same.
Hence, by Lemma \ref{lemHoloSubmersion} the pre-image $\Xi^{-1}(\mathcal{O}^{+})$
is an open and dense subset of $C^{\alpha}(\Sigma_{\mathbf{R}},\mathbb{R})$,
and for every $\varphi\in\Xi^{-1}(\mathcal{O}^{+})$ there
exists a single $\varphi$-maximizing measure and it is supported on
a periodic orbit.
\end{proof}

\section{Ergodic optimization for suspension semiflows and applications}\label{sec:reducao}

This section contains some of the key reduction arguments explored in the paper.
Throughout this section let $(N,d)$ be a compact metric space,
let $f: N \to N$ be a continuous map and $r:N\to\mathbb{R}_+$ be a continuous roof function bounded away
from zero.
Let $(X^{t})_{t\in\mathbb{R}}$ be the suspension flow over $f$ with height function $r$.
 Given $\mu\in\mathcal{M}_{1}(N^{r},(X^{t})_{t})$ denote by
$\bar{\mu}\in\mathcal{M}_{1}(N,f)$ the $f$-invariant probability measure induced by $\mu$, and recall that
$\mu=\frac{\bar{\mu}\times Leb}{\int_N r\,d\bar{\mu}}$.

\subsection{Reduction to the ergodic optimization of the Poincar\'e map}\label{sec:reduc1}

\begin{lemma}\label{eq:intphi=00003Dintphibar/intr}
Fix $\alpha >0$. If $r$ is a continuous (resp. $\al$-H\"older) roof function
and $\Phi:N^{r}\to\mathbb{R}$ is a continuous (resp. $\al$-H\"older) observable then
the induced observable $\varphi: N\to\mathbb{R}$ defined by
$
\varphi(x)=\int_{0}^{r(x)}\Phi(x,s)\,ds,
$
is a continuous (resp. $\al$-H\"older) observable and
$
\int_{{N^{r}}}\Phi \,d\mu=\frac{\int_N\varphi \,d\bar{\mu}}{\int_{N} r \,d\bar{\mu}}.
$
\end{lemma}

\begin{proof}
Since $\mu=\frac{\bar{\mu}\times Leb}{\int_N r\,d\bar{\mu}}$ then
\begin{align*}
\int_{N^r} \Phi \,d\mu
& =\int\Phi\circ\chi_{N^r}\, d\mu
  =\frac{1}{\int_{N}rd\bar{\mu}}\int_{N \times\mathbb{R}_+}\Phi\circ\chi_{N^r}(x,s)\,d\bar{\mu}\times Leb\\
 & =\frac{1}{\int_{N}r \,d\bar{\mu}}\int_{N}\int_{0}^{r(x)}\Phi(x,s)\,ds\,d\bar{\mu}
  =\frac{\int_{N}\varphi \,d\bar{\mu}}{\int_{N}r \,d\bar{\mu}}.
\end{align*}
It remains to establish the regularity of the induced observable.
Take $x,y\in N$ with $r(x)\ge r(y)$ (the case $r(x)\le r(y)$ is
analogous). Using that $\Phi$ and $r$ are continuous, we have
\begin{align}
\left|\varphi(x)-\varphi(y)\right| & =\left|\int_{0}^{r(x)}\Phi(x,s)ds-\int_{0}^{r(y)}\Phi(y,s)ds\right|\nonumber \\
 & \le\int_{0}^{r(y)}\left|\Phi(x,s)-\Phi(y,s)\right|ds+\int_{r(y)}^{r(x)}\Phi(x,s)ds\nonumber \\
 & \le\sup r\cdot\sup_{s\in(0,r(y))} |\Phi(x,s)-\Phi(y,s)|+\sup|\Phi|\cdot|r(x)-r(y)| \label{eq:phibarHolder(2.290)}
\end{align}
which ensures the continuity of $\varphi$.
If, in addition, $\Phi$ and $r$ are $\al$-H\"{o}lder continuous then one can use proceed to bound ~\eqref{eq:phibarHolder(2.290)}
and obtain
\begin{align}
\left|\varphi(x)-\varphi(y)\right|
& \le b\cdot\sup_{s\in(0,r(y))}d_{N^{r}}((x,s),(y,s))^{\alpha}+\sup|\Phi|\text{·}Ld_{N}(x,y)^{\alpha}\label{eq:phibarHolder(2.29)}
\end{align}
for some positive constants $L$ and $b$. It follows from Proposition \ref{prop:cdpi<dMr<cdpi(2.1)},
inequality (\ref{eq:phibarHolder(2.29)}) and the relation of $d_{N^{r}}$
with the pseudo metric $d_{\pi}$ expressed in (\ref{eqdpi(2.18)})
that
\begin{align*}
|\varphi(x)-\varphi(y)| & \le\sup|\Phi|\text{·}Ld_{N}(x,y)^{\alpha}+bc \, \sup_{s\in(0,r(y))} d_{\pi}((x,s),(y,s))^{\alpha}\\
 & \le[\sup|\Phi|\cdot L+bc]d_{N}(x,y)^{\alpha}.
\end{align*}
This proves the regularity of the induced observable in the H\"older category and finishes the proof of the lemma.
\end{proof}

The next result allow us to reduce the ergodic optimization of suspension semiflows to the ergodic optimization of
the continuous base dynamics, with respect to the induced observables.

\begin{lemma}
\label{le:red1} Let $(X^{t})_{t}:N^{r}\to N^{r}$ be a suspension
flow over a continuous map $f:N\to N$ on a compact metric space $N$
with continuous height function $r:N\to\mathbb{R}_+$. Let $\Phi:N^{r}\to\mathbb{R}$
be continuous and $\varphi:N\to\mathbb{R}$ be given by
$
\varphi(x)=\int_{0}^{r(x)}\Phi(x,s)\,ds.
$
The following are equivalent:
\begin{enumerate}
\item $\mu$ is a maximizing measure for $(X^{t})_{t}$ with respect to
$\Phi$;
\item $\bar{\mu}$ is a maximizing measure for $f$ with respect to $\tilde{\varphi}:=\varphi-M(\Phi,(X^{t})_{t})r$.
\end{enumerate}
In both cases, $M(\tilde{\varphi},f)=0$.
\end{lemma}

\begin{proof}
First, Lemma~\ref{eq:intphi=00003Dintphibar/intr} ensures that
\begin{align*}
M(\Phi,(X^{t})_{t}) & =\max\Big\{ \int\Phi \,d\nu \colon \nu\in\mathcal{M}_{1}(N^{r},(X^{t})_{t})\Big\}
  =\max\Bigg\{ \frac{\int_{N}\varphi \,d\bar{\nu}}{\int_{N}r \,d\bar{\nu}}\colon \bar{\nu}\in\mathcal{M}_{1}(N,f)\Bigg\}.
\end{align*}
Therefore
$
M(\Phi,(X^{t})_{t})\ge\frac{\int_{N}\varphi \,d\bar{\nu}}{\int_{N}r \,d\bar{\nu}}
$
for all $\bar{\nu}\in\mathcal{M}_{1}(N,f)$ and, consequently,
\begin{equation}\label{eq:intphibar-Mr<0}
\max_{\bar{\nu}\in\mathcal{M}_{1}(N,f)}\int_{N}\left(\varphi-M(\Phi,(X^{t})_{t})r\right)d\bar{\nu}\le0.
\end{equation}
Moreover, if $\mu$ is a maximizing measure for $(X^{t})_{t}$ with
respect to $\Phi$ and $\bar \mu$ is as before then
\begin{align*}
\int_{N}\left(\varphi-M(\Phi,(X^{t})_{t})r\right)\,d\bar{\mu} & =\int_{N}\varphi \,d\bar{\mu}-M(\Phi,(X^{t})_{t})\int_{N}r\, d\bar{\mu}\\
 & =\int_{N}\varphi \,d\bar{\mu}-\int_{N^{r}}\Phi \,d\mu\int_{N}rd\bar{\mu}
 =0.
\end{align*}
Since zero is the maximum possible value
for $\int_{N}\left(\varphi-M(\Phi,(X^{t})_{t})r\right)d\bar{\mu}$ (recall \eqref{eq:intphibar-Mr<0}),
$\bar{\mu}$ is a maximizing measure for $\tilde \varphi=\varphi-M(\Phi,(X^{t})_{t})r$
with respect to $f$ and $M(\tilde{\varphi},f)=0$.

Conversely, suppose that $\bar{\mu}$ is a maximizing measure
for $\tilde{\varphi}:=\varphi-M(\Phi,(X^{t})_{t})r$ with respect
to $f$. We claim that $M(\tilde{\varphi},f)=0$. Suppose
by contradiction that
$
M(\tilde{\varphi},f)=\max_{\bar{\nu}\in\mathcal{M}_{1}(N,f)}\int_{N}\left(\varphi-M(\Phi,(X^{t})_{t})r\right)d\bar{\nu}<0.
$
In this case,
\[
\frac{M(\tilde{\varphi},f)}{\int_{N}rd\bar{\nu}}\le\frac{M(\tilde{\varphi},f)}{\|r\|_{\infty}}<0
\]
since, for any $\bar{\nu}\in\mathcal{M}_{1}(N,f)$, $\int_{N}rd\bar{\nu}\le\|r\|_{\infty}$.
Consequently $\int_{N}\left(\varphi-M(\Phi,(X^{t})_{t})r\right)d\bar{\nu}  \le M(\tilde{\varphi},f)<0 $ implies that
\begin{align*}
& \frac{\int_{N}\varphi d\bar{\nu}}{\int_{N}rd\bar{\nu}}-\frac{M(\Phi,(X^{t})_{t})\int_{N}rd\bar{\nu}}{\int_{N}rd\bar{\nu}}  \le\frac{M(\tilde{\varphi},f)}{\int_{N}rd\bar{\nu}}<0
\end{align*}
and so
$
 \int_{N^{r}}\Phi d\nu-M(\Phi,(X^{t})_{t})  \le\frac{M(\tilde{\varphi},f)}{\|r\|_{\infty}}<0.
$
Therefore there is $a>0$ such that $\int_{N^{r}}\Phi d\nu-M(\Phi,(X^{t})_{t})<-a$
for all $\nu\in\mathcal{M}_{1}(N^{r},(X^{t})_{t})$ and taking the
maximum over $\nu$ we conclude that
$
\max_{\nu\in\mathcal{M}_{1}(N^{r},(X^{t})_{t})}\int_{N^{r}}\Phi d\nu-M(\Phi,(X^{t})_{t})<-a<0,
$
leading to a contradiction and proving the claim.
Now, if $\bar\mu$ is a maximizing measure with respect to $\tilde\varphi$
then $\int_{N}\left(\varphi-M(\Phi,(X^{t})_{t})r\right)d\bar{\mu}=0$
and, by Lemma~\ref{eq:intphi=00003Dintphibar/intr},
$
M(\Phi,(X^{t})_{t})=\frac{\int_{N}\varphi d\bar{\mu}}{\int_{N}rd\bar{\mu}}=\int_{N^{r}}\Phi d\mu,
$
which proves that $\mu$ is a maximizing measure for $\Phi$ with respect to $(X^{t})_{t}$.
\end{proof}

\subsection{Ergodic optimization: from discrete-time to continuous time}

In Subsection~\ref{sec:reduc1} we proved that one can relate the ergodic optimization
for semiflows with the one for the Poincar\'e return maps. However such a relation, expressed
by Lemma~\ref{le:red1} could \emph{a priori} relate a topologically large set of observables on the suspension space
$N^r$ with a meager set of observables on $N$.
In its essence the next lemmas will allow us to see that this is not the case and provides a correspondence between maximizing measures
for potentials on the Poincar\'e map and maximizing measures for suspension
semiflows.

\begin{lemma}
\label{lem:submersion} Fix $\al\ge 0$.
The map $\mathfrak{F}:C^{\alpha}(N^{r},\mathbb{R}) \to C^{\alpha}(N,\mathbb{R})$
given by
\begin{equation}\label{eq_submersion}
\mathfrak{F}(\Phi)=\int_{0}^{r(x)}\Phi(x,s)\,ds
\end{equation}
is a submersion.
\end{lemma}

\begin{proof}
As $\mathfrak{F}$ is linear (thus $D_{\Phi}\mathfrak{F}(H)=\mathfrak{F}(H)$
for $H\in C^{\alpha}(N^{r},\mathbb{R})$), in order to prove that $\mathfrak{F}$
is a submersion we need only prove that it is surjective. Indeed,
for any $\varphi\in C^{\alpha}(N,\mathbb{R})$ the observable
$\Phi(x,t)={\displaystyle \frac{\varphi(x)}{r(x)}}$ belongs to $C^{\alpha}(N^{r},\mathbb{R})$ (because $r(x)>0$ for every $x\in N$)
and
$
\mathfrak{F}(\Phi)  =\int_{0}^{r(x)}\Phi(x,s)\,ds
  =\int_{0}^{r(x)}{\displaystyle \frac{\varphi(x)}{r(x)}}\,ds
  =\varphi(x).
$
Therefore $D_{\Phi}\mathfrak{F}$ is surjective and $\mathfrak{F}$
is a submersion.
\end{proof}

In view of Lemma~\ref{lem:submersion} we can expect the ergodic optimization for discrete-time dynamics
to be lifted to the context of suspension semiflows.
On the one hand, the pre-image of an open  (resp. dense) set by $\mathfrak{F}$ is open (resp. dense).
However, on the other hand,
the image of $\mathfrak{F}$ does not restrict to the set of observables with maximum zero, a condition that is crucial in the relation between ergodic optimization of the flows and the Poincar\'e maps (recall item (2) in  Lemma~\ref{le:red1}).

In order to overcome this issue we
consider a further decomposition on the vector spaces of observables.
Given $\al\ge 0$ we say that a property (P) on the
maximizing measures associated of elements in
$C^\al(N,\mathbb R)$ is \emph{invariant under translation by constants} whenever the following holds:
the maximizing measures of $\varphi \in C^\al(N,\mathbb R)$ satisfy (P)
if and only if the maximizing measures of $\varphi + k \in C^\al(N,\mathbb R)$ satisfy (P) for every $k\in \mathbb R$.
For instance, the property (P) of having a unique maximizing measure is clearly invariant under translation by constants.
The following is the main result of this section.

\begin{proposition} \label{thm:susp-flow}
Let $\al\ge 0$ and let (P) be a property on the maximizing measures of the elements of $C^\al(N,\mathbb R)$
that is invariant under translation by constants.
Let $\mathcal P$ be the space of  all $\varphi \in C^\al(N,\mathbb R)$ so that all $\varphi$-maximizing measures satisfy (P).
Then the following hold:
\begin{enumerate}
\item[(i)] if $\mathcal P$ is dense in $C^\al(N,\mathbb R)$ then the set of observables $\Phi \in C^\al(N^r,\mathbb R)$
so that all $\Phi$-maximizing measures satisfy (P)  is dense in $C^\al(N^r,\mathbb R)$; and
\item[(ii)]  if $\mathcal P$ is open in $C^\al(N,\mathbb R)$ then
the set of observables $\Phi \in C^\al(N^r,\mathbb R)$
so that all $\Phi$-maximizing measures satisfy (P) is open in $C^\al(N^r,\mathbb R)$.
\end{enumerate}
\end{proposition}

\begin{proof}
Let $\al\ge 0$, let (P) be a property as above. One can write
\begin{equation}
C^{\alpha}(N,\mathbb{R})=\bigcup_{k\in\mathbb{R}}C_{k},
\quad\text{where} \quad C_{k}:=\{\psi\in C^{\alpha}(N,\mathbb{R}):M(\psi,f)=k\} \label{eq:setC_k}
\end{equation}
is a level set of maximization for every $k\in\mathbb{R}$.
Consider the map
$
\pi_{0}:C^{\alpha}(N,\mathbb{R})  \to C_{0}
$
given by $\pi_{0}(\varphi)=\varphi-M(\varphi,\sigma)$.
Our assumption implies that all $\varphi$-maximizing measures satisfy property (P) if and only if
all $\pi_{0}(\varphi)$-maximizing measures satisfy property (P).
While it is unclear if the map $\pi_0$ is a submersion (as is would require differentiability of the function $\varphi\mapsto M(\varphi,\sigma)$), the first step in the proof (Claims~\ref{c1} and~\ref{c2} below) consists of proving that the space of observables whose maximizing measures
satisfy (P) have the same structure
along the level sets defined in ~\eqref{eq:setC_k}.

\begin{claim}\label{c1}
If $\mathcal O\subset C^\al(N,\mathbb R)$ is open then $\pi_{0}(\mathcal{O})$ is open in $C_{0}$.
\end{claim}

\begin{claim}\label{c2}
If $\mathcal O\subset C^\al(N,\mathbb R)$ is dense then $\pi_{0}(\mathcal{O})$ is dense in $C_{0}$.
\end{claim}

\begin{proof}[Proof of Claim~\ref{c1}]
Take any $\varphi_{1}\in\pi_{0}(\mathcal{O})$. We will show that
$\varphi_{1}$ is an interior point of $\pi_{0}(\mathcal{O})$ in
$C_{0}$. There exists $\psi_{1}\in\mathcal{O}$ such
that $\varphi_{1}=\psi_{1}-M(\psi_{1},\sigma)$. Denote $k_{1}=M(\psi_{1},\sigma)$
and consider the set $C_{k_{1}}$, as in (\ref{eq:setC_k}). Since
$\mathcal{O}$ is open in $C^{\alpha}(N,\mathbb{R})$,
$\mathcal{O}\cap C_{k_{1}}$ is open in $C_{k_{1}}$, so there is
$\vep_{1}>0$ such that $B(\psi_{1},\vep_{1})\cap C_{k_{1}}\subset\mathcal{O}\cap C_{k_{1}}$,
where $B(\psi_{1},\vep_{1})$ is the open ball in $C^{\alpha}(N,\mathbb{R})$
with center in $\psi_{1}$ and radius $\vep_{1}$. For any $\varphi_{2}\in B(\varphi_{1},\vep_{1})\cap C_{0}$
define $\psi_{2}:=\varphi_{2}+k_{1}$. Since $M(\varphi_{2},\sigma)=0$,
we have that $M(\psi_{2},\sigma)=k_{1}$, hence $\psi_{2}\in C_{k_{1}}$
and
$
\|\psi_{1}-\psi_{2}\|=\|\psi_{1}-\varphi_{2}-k_{1}\|=\|\varphi_{1}-\varphi_{2}\|\le\vep_{1},
$
so $\psi_{2}\in B(\psi_{1},\vep_{1})$. Therefore $\psi_{2}\in C_{k_{1}}$
and $\varphi_{2}\in\pi_{0}(\mathcal{O})$. Since $\varphi_{2}$ was
taken arbitrarily, we have that $B(\varphi_{1},\vep_{1})\cap C_{0}\subset\pi_{0}(\mathcal{O})$,
which means that $\varphi_{1}$ is an interior point of $\pi_{0}(\mathcal{O})$
in $C_{0}$. Therefore $\pi_{0}(\mathcal{O})$ is an open subset of
$C_{0}$, which proves the claim.
\end{proof}

\begin{proof}[Proof of Claim~\ref{c2}]
In order to prove that $\pi_{0}(\mathcal{O})$ is dense, for any
$\varphi_{3}\in C_{0}\backslash\pi_{0}(\mathcal{O})$ we show that
$\varphi_{3}$ is a accumulation point for $\pi_{0}(\mathcal{O})$
in $C_{0}$. Since $\mathcal{O}$ is dense in $C^{\alpha}(N,\mathbb{R})$,
there is $\{\psi_{n}\}_{n}\subset\mathcal{O}$ such that $\psi_{n}\to\varphi_{3}$
as $n\to\infty$. Since $C^{\alpha}(N,\mathbb{R})\ni\varphi\mapsto M(\varphi,\sigma)$
is continuous, we have that $\pi_{0}$ is also continuous, so $\pi_{0}(\psi_{n})\to\pi_{0}(\varphi_{3})=\varphi_{3}$
as $n\to\infty$ (cf. Figure~1).
\begin{figure}[htb]\label{figura}
\includegraphics[scale=.3]{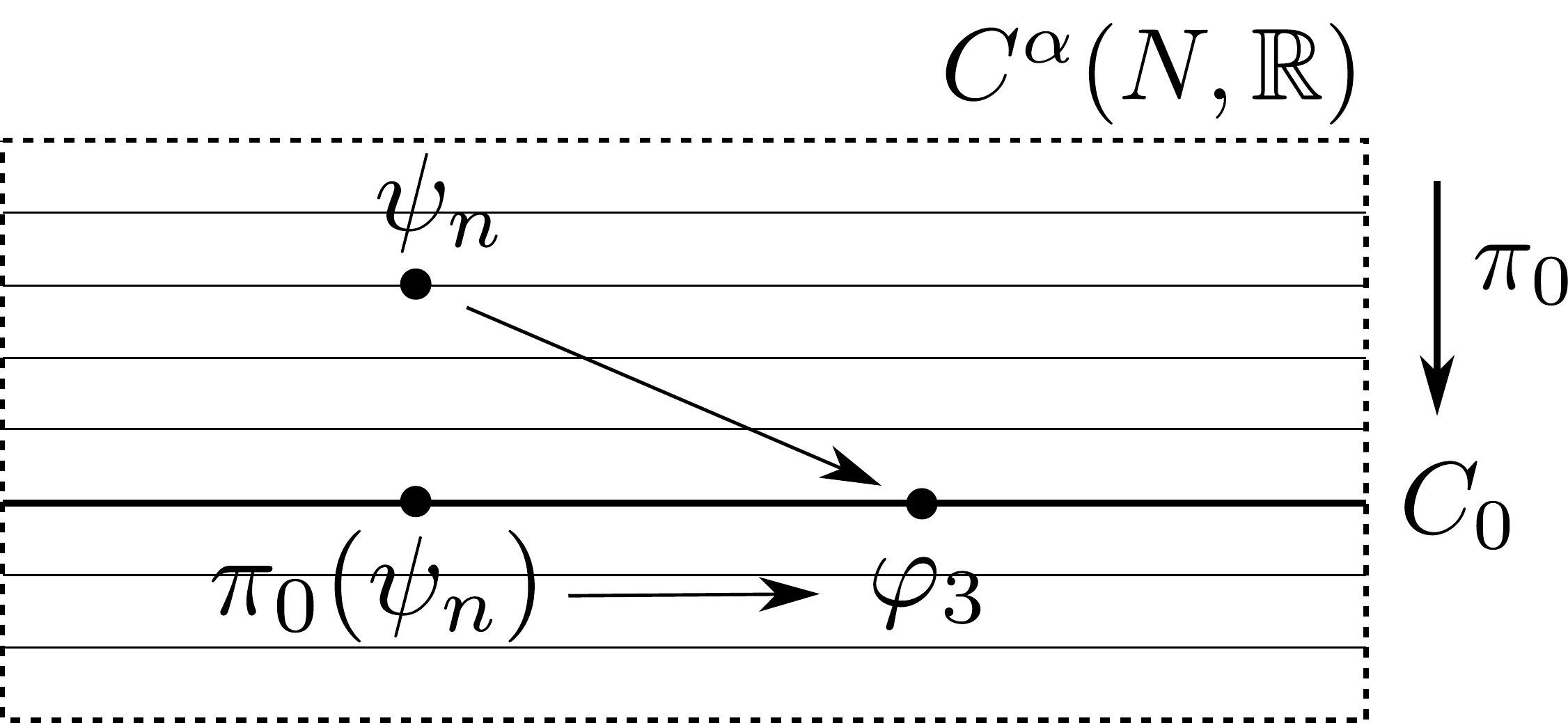}
\caption{$\psi_{n}\to\varphi_{3}\Rightarrow\pi_{0}(\psi_{n})\to\varphi_{3}$. }
\end{figure}
Therefore $\varphi_{3}$ is a accumulation point
of elements in $\pi_{0}(\mathcal{O})$, which means that $\pi_{0}(\mathcal{O})$
is dense in $C_{0}$. This proves the claim.
\end{proof}

The second step in the proof of Proposition~\ref{thm:susp-flow} is to ensure that some properties on the space
of observables lift from the context of maps to the context of semiflows.
Let $\mathfrak{F}$ be given by Lemma~\ref{lem:submersion}.
We have that the set
$
C_{0}^{r}=\{\Phi\in C^{\alpha}(N^{r},\mathbb{R}):M(\Phi,(X^{t})_{t})=0\}
$
is mapped by $\mathfrak{F}$ onto $C_0$ by $\mathfrak{F}$, that is, $\mathfrak{F}(C_{0}^{r})=C{}_{0}.$ In fact, take $\Phi\in C_{0}^{r}$, that is, $M(\Phi,(X^{t})_{t})=0$.
Writing $\tilde{\varphi}:=\varphi-M(\Phi,(X^{t})_{t})r$, where $\varphi=\mathfrak{F}(\Phi)$,
we have that $\tilde{\varphi}=\varphi$ and $M(\varphi,\sigma)=0$,
by Lemma \ref{le:red1}.

Now take $\varphi\in C_{0}$, meaning that  $M(\varphi,\sigma)=0$. Since
$\mathfrak{F}$ is surjective, there exists $\Phi\in C^{\alpha}(\Sigma_{\mathbf{R}}^{r},\mathbb{R})$
such that $\varphi=\int_{0}^{r(x)}\Phi(x,s)\,ds$. We claim that
$M(\Phi,(X^{t})_{t})=0$. Let $\mu_{\varphi}$ be $\sigma$-invariant such that $\int\varphi \,d\mu_{\varphi}=M(\varphi,\sigma)=0$
and let $\mu_{\Phi}$ be $(X^t)_t$-invariant such that $\int\Phi \, d\mu_{\Phi}=M(\Phi,(X^{t})_{t})$.
If $\overline{\mu}_{\Phi}$ is the unique $\sigma$-invariant probability so that $\mu_{\Phi}=\frac{\overline{\mu}_{\Phi}\times Leb}{\int r\, d\overline{\mu}_{\Phi}}$, Lemmas \ref{eq:intphi=00003Dintphibar/intr} and \ref{le:red1} ensure that
\begin{equation}\label{eq:Mquoc}
M(\Phi,(X^{t})_{t})=\int\Phi \,d\mu_{\Phi}=\frac{\int\varphi \, d\overline{\mu}_{\Phi}}{\int r \, d\overline{\mu}_{\Phi}}
\end{equation}
that $\int [\varphi-M(\Phi,(X^{t})_{t})r] \,d\overline{\mu}_{\Phi}=0$. Now, equation ~\eqref{eq:Mquoc} together with
the fact that  $\int\varphi \,d\overline{\mu}_{\Phi}\le M(\varphi,\sigma)=0$ implies $M(\Phi,(X^{t})_{t})\le0$.
Thus $\varphi-M(\Phi,(X^{t})_{t})r\ge\varphi$ and
\[
0= \int [\varphi-M(\Phi,(X^{t})_{t})r] \,d\overline{\mu}_{\Phi}
	\ge \int [\varphi-M(\Phi,(X^{t})_{t})r] \,d{\mu}_{\varphi}
	\ge \int\varphi d\mu_{\varphi}=0.
\]
This ensures that $\int [\varphi-M(\Phi,(X^{t})_{t})r] \,d{\mu}_{\varphi}=0$ and, consequently,
\[
M(\Phi,(X^{t})_{t})=\frac{\int\varphi d\mu_{\varphi}}{\int rd\mu_{\varphi}}=0.
\]
This shows that $\mathfrak{F}(C_{0}^{r})\supset C{}_{0}$.
Moreover, since $\mathfrak{F}$ is a submersion,
Claims~\ref{c1} and
~\ref{c2} ensure that the pre-image $\mathfrak{F}^{-1}(\pi_0(\mathcal{O}))$ is
an open (resp. dense) subset of $C_{0}^{r}$ whenever $\mathcal O$ is an open (resp. dense)
subset of $C^\al(N,\mathbb R)$. Now consider the projection $\Pi_0: C^{\alpha}(N^{r},\mathbb{R}) \to C_0^r$
defined by $\Phi \mapsto \Phi-M(\Phi,(X^{t})_{t})$.
Since the maximizing measure of every element in
$\mathfrak{F}^{-1}(\pi_0(\mathcal{O}))$ satisfies property (P) then we are left to prove the following:

\begin{claim}
\label{claim:open_and_dense} If $\mathcal{O} \subset C^{\alpha}(N,\mathbb{R})$ is open (resp. dense) then
$\Pi_0^{-1}(\mathfrak{F}^{-1}(\pi_0(\mathcal{O})))$ is an open (resp. dense) subset of $C^{\alpha}(N^{r},\mathbb{R})$.
\end{claim}

\begin{proof}
In order to prove that  $\Pi_0^{-1}(\mathfrak{F}^{-1}(\pi_0(\mathcal{O})))$ is open we prove that any
$\Phi \in \Pi_0^{-1}(\mathfrak{F}^{-1}(\pi_0(\mathcal{O})))$ in an interior point.
Let $\Phi_{0}:=\Pi_0(\Phi)=\Phi-M(\Phi,(X^{t})_{t})\in\mathfrak{F}^{-1}(\pi_0(\mathcal O))$.
Since $\mathfrak{F}^{-1}(\pi_0(\mathcal{O}))$ is open in $C_{0}^{r}$, there exists $\vep>0$ such that if $\Upsilon\in C_{0}^{r}$ and
$\|\Phi_{0}-\Upsilon\|<\vep$ then $\varUpsilon\in\mathfrak{F}^{-1}(\pi_0(\mathcal{O}))$.
On the other hand, since $\Phi\mapsto M(\Phi,(X^{t})_{t})$ is a continuous
map, there is $0<\delta<\vep/2$ such that if $\|\Phi-\Psi\|<\delta$, then
$|M(\Phi,(X^{t})_{t})-M(\Psi,(X^{t})_{t})|<{\displaystyle \frac{\vep}{2}}$.
So if $\Psi\in C^{\alpha}(N^{r},\mathbb{R})$
satisfies $\|\Phi-\Psi\|<\delta$ then
\begin{align*}
\|\Phi_{0}-\Psi_{0}\| & =\|\Phi-M(\Phi,(X^{t})_{t})-\Psi+M(\Psi,(X^{t})_{t})\|\\
 & \le\|\Phi-\Psi\|+\|M(\Phi,(X^{t})_{t})-M(\Psi,(X^{t})_{t})\|
  <\delta+\frac{\vep}{2}
  <\vep
\end{align*}
which means that $\Psi_{0}\in C_{0}^{r}$. So $\Psi\in\Pi_0^{-1}(\mathfrak{F}^{-1}(\pi_0(\mathcal{O})))$
and consequently $\Pi_0^{-1}(\mathfrak{F}^{-1}(\pi_0(\mathcal{O})))$ is open.
In order to show that $\Pi_0^{-1}(\mathfrak{F}^{-1}(\pi_0(\mathcal{O})))$ is dense,
we take any $\Psi\in C^{\alpha}(N^{r},\mathbb{R})\backslash \Pi_0^{-1}(\mathfrak{F}^{-1}(\pi_0(\mathcal{O})))$
and show that $\Psi$ is a accumulation point of elements in $\Pi_0^{-1}(\mathfrak{F}^{-1}(\pi_0(\mathcal{O})))$.
Since $\mathfrak{F}^{-1}(\pi_0(\mathcal{O}))$ is dense in $C_{0}^{r}$,
for any $\vep>0$ there is $\Upsilon\in\mathfrak{F}^{-1}(\mathcal{O}_{0})$
such that $\|\Psi_{0}-\Upsilon\|<\vep$. Taking $\Phi=\Upsilon+M(\Psi,(X^{t})_{t})$,
note that $M(\Phi,(X^{t})_{t})=M(\Psi,(X^{t})_{t})$.
 Moreover $\Phi\in \Pi_0^{-1}(\mathfrak{F}^{-1}(\pi_0(\mathcal{O})))$, because $M(\Phi,(X^{t})_{t})=M(\Psi,(X^{t})_{t})$
and
\begin{align*}
\Phi_{0}  =\Phi-M(\Phi,(X^{t})_{t})
  =\Upsilon+M(\Psi,(X^{t})_{t})-M(\Psi,(X^{t})_{t})
  =\Upsilon\in\mathfrak{F}^{-1}(\pi_0(\mathcal{O})).
\end{align*}
We also have
$\|\Psi-\Phi\|=\|\Psi-\Upsilon-M(\Psi,(X^{t})_{t})\|=\|\Psi_{0}-\Upsilon\|<\vep.$
Since $\vep$ was arbitrary the density of $\Pi_0^{-1}(\mathfrak{F}^{-1}(\pi_0(\mathcal{O})))$ follows. This finishes
the proof of the claim.
\end{proof}
The proof of the proposition is now complete.

\end{proof}

\subsection{Suspension flows over subshifts of finite type}

The following result summarizes combines
Proposition~\ref{thm:susp-flow} together with the
results on ergodic optimization for subshifts of finite type in Subsection~\ref{unilateral-bilateral}.

\begin{theorem} \label{thm:semi-flow} Let $\sigma:\Sigma_{\mathbf{R}}\to\Sigma_{\mathbf{R}}$
be a transitive two-sided subshift of finite type, let  $r:\Sigma_{\mathbf{R}}\to\mathbb{R}$ be a continuous
roof function bounded away from zero and let $(X^{t})_{t\in\mathbb{R}}$
be the suspension flow over $\sigma$ with height function $r$. Then the following hold:
\begin{enumerate}
\item  exists an open and dense subset $\mathcal{O}\subset C^{0}(\Sigma_{\mathbf{R}},\mathbb{R})$
so that every $\Phi\in\mathcal{O}$ has a unique $(X^{t})_{t}$-maximizing
measure with respect to $\Phi$, it has zero entropy and full support;
\item exists a dense set $\mathcal{D}\subset C^{0}(\Sigma_{\mathbf{R}},\mathbb{R})$
so that every $\Phi\in\mathcal{D}$ has uncountably many ergodic $\Phi$-maximizing measures;
\item if, in addition, $r$ is H\"{o}lder continuous then
exists an open and dense set $\mathcal{R}_{r}\subset C^{\alpha}(\Sigma_{\mathbf{R}}^{r},\mathbb{R})$
of observables $\Phi:\Sigma_{\mathbf{R}}^{r}\to\mathbb{R}$ so that every
$\Phi\in\mathcal{R}_{r}$ has a unique $(X^{t})_{t}$-maximizing
measure with respect to $\Phi$, and it is supported on a periodic orbit;
\end{enumerate}
\end{theorem}

\begin{proof}
On the one hand, by Proposition \ref{prop:open_dense_shift_bilateral},
there exists an open and dense set $\mathcal{O}\subset C^{\alpha}(\Sigma_{\mathbf{R}},\mathbb{R})$
such that every $\varphi\in\mathcal{O}$ has a unique $\varphi$-maximizing
measure $\bar{\mu}$ and it is supported on a periodic orbit.
On the other hand, since $\sigma$ is a transitive subshift of finite type then
it satisfies the gluing orbit property \cite{BV}. Thus Theorem~\ref{morris} ensures that there exists a dense $G_\delta$ subset $\mathcal{Z} \subset
C^{0}(\Sigma_{\mathbf{R}},\mathbb{R})$ so that every $\varphi \in \mathcal Z$ has a unique $\varphi$-maximizing measure, with
zero entropy and full support.
Finally, the results in Subsection~\ref{unilateral-bilateral} to obtain conclusions for bilateral subshifts of finite type together with Theorem~\ref{Shinoda} ensure that exists a dense subset $\mathcal D \subset C^0(\Sigma_R,\mathbb R)$ such that for every
$\varphi \in \mathcal D$ there exist uncountably many ergodic $\varphi$-maximizing measures with full support and positive entropy.
Hence, the corollary is a direct consequence of Theorem~\ref{thm:susp-flow}.
\end{proof}

\subsection{Ergodic optimization on hyperbolic sets}

Let $\Lambda\subset M$ be a horseshoe
for the flow $(X^{t})_{t\in\mathbb{R}}$.
Since hyperbolic flows admit Markov partitions, these are modeled by suspension flows \cite{Bowen,ratner1973markov}. Then
Theorem~\ref{TeoA:hyper-subset} follows from corresponding result for suspension flows (cf. Theorem~\ref{thm:semi-flow}).
Indeed, there is a subshift of finite type $\sigma_{\mathbf{R}}:\Sigma_{\mathbf{R}}\to\Sigma_{\mathbf{R}}$,
a positive $r\in C^{\alpha}(\Sigma_{\mathbf{R}},\mathbb{R})$ and
a H\"{o}lder continuous
homeomorphism $\pi:\Sigma_{\mathbf{R}}^{r}\to\Lambda$ so that
\begin{equation}
\pi \circ Y^{t} = X^t\circ \pi
\quad \text{for every $t\in \mathbb R$,}
\label{eq:conjugsigflow-1}
\end{equation}
where $\Sigma_{\mathbf{R}}^{r}$ is a quotient as in Subsection~\ref{s:semiflows}
and $(Y^{t})_{t}:\Sigma_{\mathbf{R}}^{r}\to\Sigma_{\mathbf{R}}^{r}$
is the suspension flow over $\sigma_{\mathbf{R}}$ with height function
$r$.
If $\Lambda$ is a horseshoe then $\pi:\Sigma_{\mathbf{R}}^{r}\to\Lambda$
is one-to-one. So given an observable $\Phi\in C^{\alpha}(\Lambda,\mathbb{R})$
one can induce an observable $\Phi^{*}\in C^{\alpha}(\Sigma_{\mathbf{R}}^{r},\mathbb{R})$
by doing $\Phi^{*}=\Phi\circ\pi$ and the map $\Theta:C^{\alpha}(\Lambda,\mathbb{R})\to C^{\alpha}(\Sigma_{\mathbf{R}}^{r},\mathbb{R})$
defined by $\Theta(\Phi)=\Phi\circ\pi$ is one-to-one.
Theorem \ref{TeoA:hyper-subset} is now a direct consequence of Theorem~\ref{thm:semi-flow}.

\section{Application to Lorenz-like attractors }\label{secLorenz}

In this section we use  Theorem \ref{TeoA:hyper-subset} to prove Theorem \ref{TeoALorenz} by an approximation method.
The proof of the latter also relies on auxiliary results on the structure of invariant subsets of Lorenz attractors and the space of
invariant probability measures for wild Lorenz attractors.

\subsection{Proper subsets and the space of invariant probabilities }\label{subsecLorenz}

Let $N$ be a three-dimensional closed Riemannian manifold,
$\Lambda\subset N$ be a
geometric Lorenz attractor for the $C^1$-flow
$(X^t)_t$ generated by a $C^1$-vector field $X$ and let $\mathcal W\supset \Lambda $ denote an attracting region. Hence $\Lambda= \bigcap_{t\ge 0} X^{-t}(\mathcal W)$.
For every open neighborhood $\mathcal U$ of $\text{Sing}(X)$ and every $\varphi \in C^0(N,\mathbb R)$, and
let the maximal invariant set $ \Lambda_{\mathcal U}$ in $\mathcal W \setminus \cU$
and the constrained ergodic optimization maximum be defined by
\begin{equation}\label{eq:constrainedM}
\Lambda_{\mathcal U} := \bigcap_{t\ge 0} X^{-t}(\mathcal W \setminus \cU)
\qquad\text{and}\qquad
M_\cU(\varphi)=\max_{\substack{\mu(\Lambda_\cU)=1\\ \mu\, \text{ergodic}}}
	\Bigg\{ \int \varphi \, d\mu\Bigg\},
\end{equation}
respectively. The set $\Lambda_{\mathcal U}$ is a compact $(X^t)_t$-invariant singular-hyperbolic
 set without singularities, hence it is a uniformly hyperbolic set (cf. \cite[Proposition~6.2]{araujo2010three}).
In the case of more general singular-hyperbolic attractors \emph{a priori} the set $\Lambda_{\mathcal U}$ could be non-transitive.
We prove this is not the case for geometric Lorenz attractors.

\begin{lemma}\label{le:transitive}
Assume that $\Lambda$ is a geometric Lorenz attractor for the flow $(X^t)_t$.
If $\mathcal U$ is any open neighborhood of $\sigma$ then the invariant set $\Lambda_{\mathcal U}$ defined by
\eqref{eq:constrainedM} is transitive.
\end{lemma}

\begin{proof}
Let $\Sigma$ denote the global cross-section to the flow $(X^t)_t$. Observe that it is enough to
prove that the Poincar\'e return map $P$ defined in ~\eqref{eq:defPoincare} is transitive on $\Lambda_{\mathcal U} \cap \Sigma$.
A direct proof of this single fact seems not easy.
We prove transitivity by proving that $\Lambda_{\mathcal U} \cap \Sigma$ is a homoclinic class for $P$ (transitivity
follows as a consequence of Birkhoff-Smale theorem).

Since periodic points are dense, in order to prove that $\Lambda_{\mathcal U} \cap \Sigma$ is a homoclinic class for $P$ we claim that any two periodic
points for $P$ are heteroclinically related.  Let $p,q$ be any periodic points for $P$ (these are hyperbolic periodic points
and we assume, without loss of generality, that $p,q$ are fixed points for $P$).
Let $W_{loc}^u(p) \subset \Sigma$ denote the local Pesin unstable manifolds at $p$ and set $I=\pi^s( W_{loc}^u(p) )
\subset [-1,1]\setminus\{0\}$.
Since one-dimensional Lorenz map $\alpha$ is locally eventually onto (recall Subsection~\ref{s:shyp}), there exists $N\ge 1$
so that $\alpha^N(I) =(\alpha(-1),\alpha(1))$. This proves that $P^N(W_{loc}^u(p))$ crosses $\Sigma$, hence $W^u(p)$
intersects the stable manifold $W_{loc}^s(q)$. Replacing the roles of $p$ and $q$ above we conclude that $p$ and $q$
are heteroclinically related. Therefore, $\Lambda_{\mathcal U} \cap \Sigma$ is transitive.
\end{proof}

We proceed to describe the maximum value in ergodic optimization for flows.
By the ergodic decomposition theorem the maximum is attained at ergodic measures and, consequently,
the following equalities hold:
\begin{align}
M(\Phi,(X^{t})_{t})
	& = \max_{\substack{\mu(\Lambda)=1\\ \mu\, \text{ergodic}}}
	\Big\{ \int \Phi \, d\mu\Big\} \nonumber \\
	& = \max \Big\{
	\sup_{\substack{\mu\, \text{non-atomic}\\ \mu\, \text{ergodic}}}  \Big\{ \int \Phi \, d\mu\Big\},
	\sup_{\substack{\mu\, \text{atomic}\\ \mu\, \text{ergodic}}}
	\Big\{ \int \Phi \, d\mu\Big\}
	\Big\} \label{max12} \\
	& = \max \Big\{
	\sup_{\substack{\supp \mu=\Lambda \\ \mu\, \text{ergodic}}}  \Big\{ \int \Phi \, d\mu\Big\},
	\sup_{\substack{\supp \mu\neq \Lambda \\ \mu\, \text{ergodic}}}
	\Big\{ \int \Phi \, d\mu\Big\}
	\Big\} \label{max123} \\
		& = \max \Big\{
	\sup_{\substack{\text{Sing}(X) \cap \supp \mu \neq \emptyset \\ \mu\, \text{ergodic}}}  \Big\{ \int \Phi \, d\mu\Big\},
	\sup_{\substack{\text{Sing}(X) \cap \supp \mu = \emptyset \\ \mu\, \text{ergodic}}}
	\Big\{ \int \Phi \, d\mu\Big\}
	\Big\}. \label{max1234}
\end{align}
In order to simplify notation we used $M(\Phi,(X^{t})_{t})$ instead of $M(\Phi,(X^{t})_{t},\Lambda)$ to denote the maximum value in ergodic optimization  on the Lorenz attractor $\Lambda$.
Using Theorem~\ref{Jenkinson} we conclude that exactly one of the supremum terms in the right hand of each of the
equations \eqref{max12} - \eqref{max1234} is attained by some ergodic probability measure in the case of
typical observables. Hence it makes natural to ask whether
probabilities determined by periodic orbits are dense in the space of invariant measures for the Lorenz attractor.
The next lemma gives a positive answer to this question.

\begin{lemma}\label{le:mper}
Assume that $\Lambda$ is a wild Lorenz attractor.
The space of periodic probabilities is dense in the space $\mathcal M_e((X_t)_t)$ of $(X_t)_t$-invariant
probability measures on the attractor.
\end{lemma}

\begin{proof}
By the ergodic decomposition theorem (see e.g. \cite{KH}) it is enough to prove that every ergodic probability
can be weak$^*$ approximated by probabilities supported on periodic points.
Since this is the case for every non-atomic ergodic probability
(recall Remark~\ref{rmk:Xueting}) it remains to prove that the Dirac measure $\delta_\sigma$ is also weak$^*$ approximated by periodic probabilities.

Recall that the Lorenz attractor is modeled by a suspension flow. Moreover, changing the cross-section if
necessary, one can assume that the
return time function $\varrho$ is constant along the strong-stable leaves (see e.g. \cite{ArVar})
and
\begin{align}\label{eq:varrho-x}
  \varrho(x)= -\frac1{\lambda_1} \log|x|+s(x)
\end{align}
for some bounded function $s$. In particular, since $\alpha'(x)=\gamma x^{1-\gamma}$ in a neighborhood of $0$
(recall \eqref{eq:defLorenz}) then
there are constants $C_1,C_2>0$ so that
\begin{align}\label{eq:varrho-2}
C_1 \log |\alpha'(x)| \le   \varrho(x)  \le  C_2 \log |\alpha'(x)|
\quad\text{for all $x\neq 0$.}
\end{align}

We claim that any sequence $(\mu_n)_n$ of $P$-invariant ergodic probability measures such that $\int \log |\alpha'(x)| \, d\mu_n \to \infty$
as $n\to \infty$ converges in the weak$^*$ topology to the Dirac measure $\delta_\sigma$.
Given $\vep>0$ arbitrary let $V_\vep$ be the ball of radius $\vep$ around $\sigma$. Choose a local cross-section
$\Sigma_\vep \subset \Sigma$
 so that the following hold:
\begin{itemize}
\item[(i)] the piece of orbit $\{X^t(x) \colon 0\le t\le \varrho(x)\}$ of a point $x\in \Sigma_\vep$ intersects $V_\vep$, and
\item[(ii)] $\Leb\, \Big(0\le t\le \varrho(x) \colon X^t(x)\in V_\vep\Big) > (1-\vep) \, \varrho(x)  $ for every $x\in \Sigma_\vep$.
\end{itemize}
The construction of a local cross-section $\Sigma_\vep$ satisfying (i) and (ii) is possible taking $\Sigma_\vep$ as a small neighborhood of the stable manifold of $W^s(\sigma)\cap \Sigma$  and noting that, due to the continuous dependence of solutions of ordinary differential equations
and that the return time function can be chosen arbitrary large for points close to the stable manifold, which ensures that these
orbits spend as much proportion of time  near the singularity
as desired.

By construction $\Sigma_\vep$ is a small neighborhood of the stable manifold of $W^s(\sigma)\cap \Sigma$,
hence the roof function $\varrho$ is bounded by some constant $C_\vep$ all points in $\Sigma\setminus \Sigma_\vep$. Relation \eqref{eq:varrho-2}
implies that
$
\int_{\Sigma_\vep} \varrho \, d\mu_n \to \infty
\;\text{as}\; n\to\infty.
$
For any $T>0$ let $N(x,T)\ge 1$ be the unique integer such that
$\sum_{j=0}^{N(x,T)-1} \varrho(P^j(x)) \le T < \sum_{j=0}^{N(x,T)} \varrho(P^j(x))$.
Then, Birkhoff's ergodic theorem ensures that for $\mu_n$-a.e. $x$
$$\frac{T}{N(x,T)} \to \int \varrho \, d\mu_n
\quad\text{as} \;T\to\infty$$
and, consequently,
\begin{align}\label{eq:convDirac}
\mu_n(V_\vep)=\lim_{T\to\infty} \frac1T \Leb\, \Big(0\le t\le T \colon X^t(x)\in V_\vep\Big)
	& \ge \lim_{T\to\infty}  \frac{N(x,T)}{T} \cdot \frac{1}{N(x,T)} \sum_{j=0}^{N(x,T)-1} (1-\vep) (\varrho \cdot \chi_{\Sigma_\vep} )(P^j(x))
	\nonumber \\
	& = (1-\vep)\, \frac{\int_{\Sigma_\vep} \varrho \, d\mu_n}{\int \varrho \, d\mu_n}
	\ge (1-\vep)\, \Big[1- \frac{C_\vep}{\int \varrho \, d\mu_n}\Big]
\end{align}
for $\mu_n$-almost every $x\in \Sigma$. Note that the right hand-side of equation~\eqref{eq:convDirac} tends to $1-\vep$ as $n\to\infty$.
In particular, since $\vep$ was taken arbitrary,
 any accumulation point $\mu$ of $(\mu_n)_n$ satisfies $\mu(\overline{V_\vep}) \ge 1-\vep$ for every $\vep>0$.
Hence, we conclude that $\mu_n \to\delta_\sigma$ in the weak$^*$ topology. This proves the lemma.
\end{proof}

\begin{remark}
In the case of singular-hyperbolic attractors, all singularities and periodic orbits are hyperbolic. In particular, for every $T>0$
the set of critical elements formed by singularities and periodic orbits of period smaller than $T$ consists of finitely many disjoint orbits.
Since there are countably many critical elements, it is not hard to check that there exists a $C^0$-residual subset of $C^0(M,\mathbb R)$ formed by observables $\varphi$ such that
$$\int \varphi \, d\mu_p \neq \int \varphi \, d\mu_q$$ for every distinct $p,q\in \text{Crit}(X)$, where $\mu_p=\delta_p$ if $p$ is a singularity
and $\mu_p=\frac1{\pi(p)} \int_0^{\pi(p)} \delta_{X^s(p)} \, ds$ whenever $p$ is a periodic point of period $\pi(p)>0$.
An analogous statement also holds for H\"older continuous observables.
\end{remark}

\subsection{Ergodic optimization for continuous observables }

The following lemma gives the starting point to the ergodic optimization
of continuous observables, and it ensures that maximizing measures for $C^0$-typical observables on singular-hyperbolic attractors
(hence for Lorenz attractors) are not supported at singularities.

\begin{lemma}\label{le:nonsingularm}
Let $\Lambda$ be a singular-hyperbolic attractor for $(X^t)_t$.
The set $$ {\mathcal O}=\Big\{\Phi \in  C^0(N,\mathbb R) \colon \int \Phi \, d\delta_\sigma < M(\Phi,(X^{t})_{t}), \; \, \forall \sigma \in \text{Sing}(X)\Big\}$$ is a $C^0$-open and dense subset of $C^0(N,\mathbb R)$.
\end{lemma}

\begin{proof}
Since the singularities of singular-hyperbolic attractors are hyperbolic, there are finitely many of them, which we denote by
$\{\sigma_1, \sigma_2, \dots, \sigma_n\}$, and it is clear that $\mathcal O$ is $C^0$-open.

It remains to prove that $\mathcal O$ is $C^0$-dense.
Assume that $\Phi \in  C^0(N,\mathbb R)$ and that there exists ${\sigma_i} \in \text{Sing}(X)$ so that
$\int \Phi \, d\delta_{\sigma_i} = M(\Phi,(X^{t})_{t})$.
We use that for any singularity ${\sigma_i}\in \Lambda$ the probability measure $\delta_{\sigma_i}$ is accumulated by invariant
and ergodic measures associated to periodic measures $(\mu_{p})_{p\in Per((X^t)_t)}$ (recall Lemma~\ref{le:mper}).
In particular,
$$
M(\Phi,(X^{t})_{t})= \int \Phi \, d\delta_{\sigma_i} = \sup_{p \in \text{Per}((X^t)_t)} \int \Phi \, d\mu_p.
$$
For any $\vep>0$ let $p \in \text{Per}((X^t)_t)$ be such that
$\int \Phi \, d\mu_p > \int \Phi \, d\delta_{\sigma_i} -\frac\vep{2}$ and let $0<\delta=\delta({\sigma_i},p) \ll \vep$ be such that
$\min_{t\in [0,\pi(p)]} \dist({\sigma_i}, X^t(p)) \ge 2\delta>0$. Performing a $\vep$-$C^0$-small perturbation supported
in a $\delta$-neighborhood of $p$ we obtain a $C^0$-observable $\Phi_1$ so that $\Phi_1 \equiv \Phi$ on $M\setminus B({\sigma_i},\delta)$ and
$\Phi_1({\sigma_i}) < \Phi({\sigma_i})-\vep$. In particular
$$
\int \Phi_1 \, d\delta_{\sigma_i} < \int \Phi \, d\delta_{\sigma_i} -\vep < \int \Phi \, d\mu_p -\frac\vep2 = \int \Phi_1 \, d\mu_p
-\frac\vep2  \le M(\Phi_1) -\frac\vep2,
$$
hence $\delta_{\sigma_i}$ is not a $\Phi_1$-maximizing measure. Since there are finitely many singularities, after
a finite number of $C^0$-perturbations with disjoint supports we obtain a $C^0$-observable $\Phi_k$ that is $\vep$-$C^0$-close
to $\Phi_k$ for which no Dirac measure at a singularity is a
$\Phi_k$-maximizing measure. This shows that $\Phi_k\in \mathcal O$ and proves the lemma.
\end{proof}

Let $\mathcal O$ be given by the previous lemma. Using Theorem~\ref{Jenkinson} together with  Remarks~\ref{rmk:Holder} and ~\ref{rmk:thm-unique-flows}, we conclude that there exists a $C^0$-residual subset $\mathfrak R_1 \subset C^0(N,\mathbb R)$ so that every
$\Phi\in \mathcal R$ has a unique maximizing measure. Consider the Baire residual subset
$\mathcal R_1:=\mathfrak R_1 \cap \mathcal O$.  By construction, every  $\Phi \in \mathcal R_1$ has a unique maximizing measure $\mu$ (hence ergodic). Item (1) of Theorem~\ref{TeoALorenz} and the following characterization of the support of maximizing measures:

\begin{lemma}\label{le:ntom}
Let $\Lambda$ be a Lorenz attractor.
There exists a $C^0$-Baire generic subset of $C^0(N,\mathbb R)$ formed by observables
which have a unique maximizing measure $\mu$, which is non-atomic and so that $\supp \mu$ contains a singularity.
\end{lemma}

\begin{proof}
Let $\Phi\in \mathcal R_1$ and let $\mu$ be the unique (ergodic) $\Phi$-maximizing measure.
Since $\mathcal R_1 \subset \mathcal O$ then $\mu$ cannot be the Dirac mass at the singularity.
In order to conclude that $\mu$ is non-atomic it is enough to show that the support of $\mu$ contains the singularity.
This requires extra information, of independent interest, on the relative maximization at hyperbolic sets as follows.

For any $\vep>0$ let $\mathcal U_\vep$ denote the $\vep$-neighborhood of $\text{Sing}(X)$.
We describe three situations to consider for the shape of the non-increasing function $\vep \to M_{\cU_\vep}(\Phi)$
(see Figure~2).
\begin{figure}[htb]
 \subfigure{
    \includegraphics[scale=0.25]{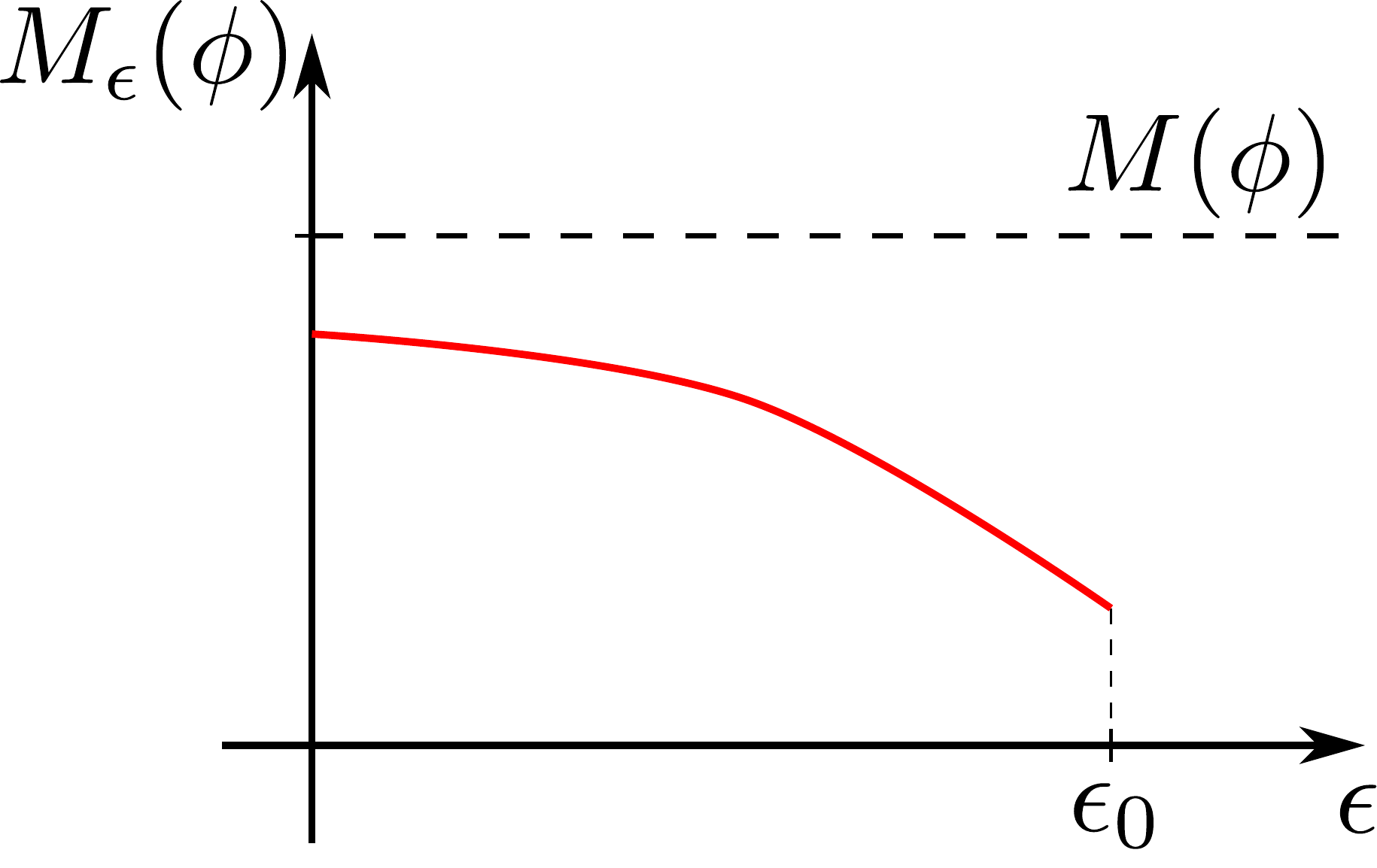}
  }
  \subfigure{
    \includegraphics[scale=0.25]{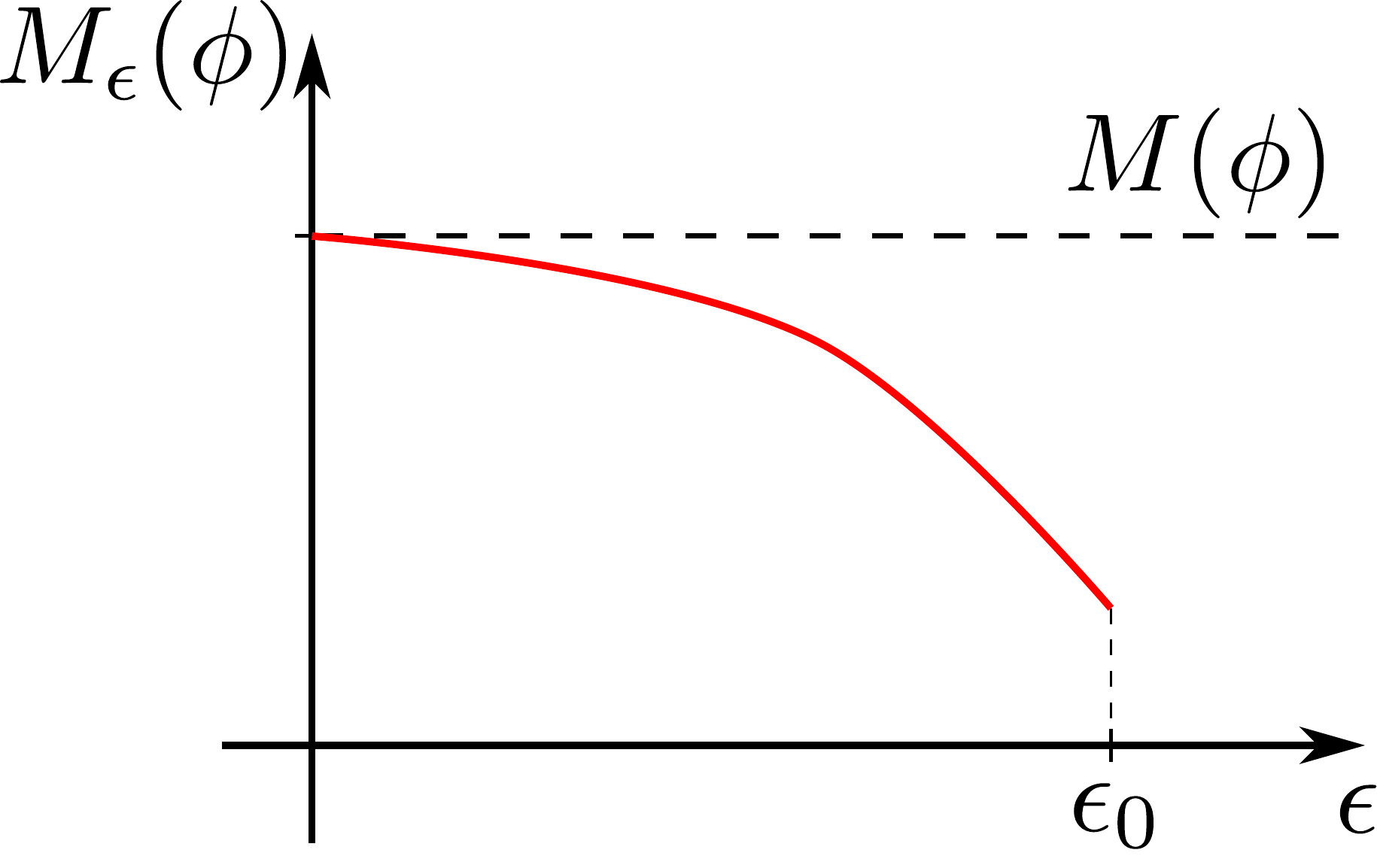}
  }
  \subfigure{
    \includegraphics[scale=0.25]{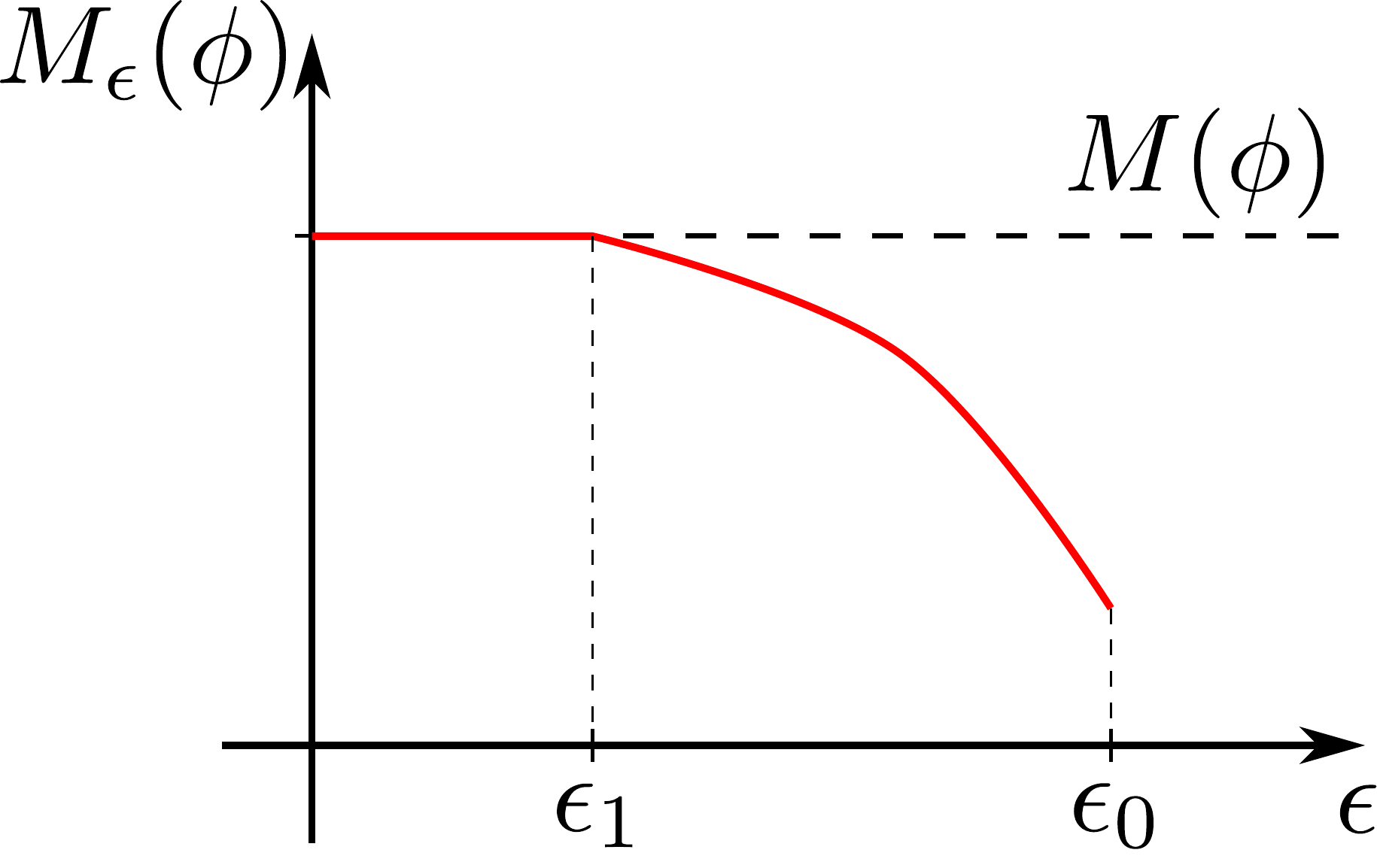}
  }
\caption{Possible shapes for the function $\vep \to M_{\cU_\vep}(\Phi)$ }
\end{figure}
In cases a. and b. we have $M_{\cU_\vep}(\Phi) < M(\Phi)$  for every $\vep>0$.
If $\supp\mu \subset \Lambda$ contains no singularity then there exists $\vep_0>0$ such that
$\supp\mu \cap \cU_{\vep_0} =\emptyset$, thus $\supp\mu \subset \Lambda_{\cU_{\vep_0}}$
contradicting the fact that $M_{\cU_{\vep_0}}(\Phi) < M(\Phi)$. Thus proves that
$\supp\mu \subset \Lambda$ contains a singularity.

We claim that case c. does not occur for $C^0$-generic observables. Indeed, let $\Phi\in C^0(N,\mathbb R)$
and $\vep_1>0$ be so that $M_{\cU_\vep}(\Phi) =M(\Phi,(X^{t})_{t})$ for every $0<\vep\le \vep_1$.
Pick $0<\vep_2<\vep_1$ such that $\Lambda_{\cU_{\vep_2}} \subsetneq \Lambda_{\cU_{\vep_1}}$.
Using that each map $C^0(N,\mathbb R) \to C^0(\Lambda_{\cU_{\vep_i}},\mathbb R)$ given by
$\Phi \mapsto \Phi\mid_{\Lambda_{\cU_{\vep_i}}}$ is a submersion ($i=1,2$)
and Theorem~\ref{TeoA:hyper-subset} we deduce that
there exist $C^0$-Baire residual subsets $\mathcal R_{\vep_1}, \mathcal R_{\vep_2} \subset C^0(N,\mathbb R)$ such that
every $\Psi\in \mathcal R_i$ has a unique $\Psi$-maximizing measure on $\Lambda_{\cU_{\vep_i}}$, and its support coincides
with $\Lambda_{\cU_{\vep_i}}$ ($i=1,2$).
Therefore, any observable $\Psi$ in the
$C^0$-Baire residual subset  $\mathcal R_1 \cap \mathcal R_{\vep_1}\cap \mathcal R_{\vep_2}$  has a unique maximizing
measure on $\Lambda$, $\Lambda_{\cU_{\vep_2}}$ and $\Lambda_{\cU_{\vep_1}}$. Since
$M_{\cU_{\vep_1}}(\Phi) =M_{\cU_{\vep_2}}(\Phi) =M(\Phi,(X^{t})_{t})$ we conclude that $\mu_{\vep_2}=\mu_{\vep_1}$,
which contradicts the fact that $\Lambda_{\cU_{\vep_2}} \subsetneq \Lambda_{\cU_{\vep_1}}$.
Thus, $C^0$-generic observables do not satisfy case c.
This finishes the proof of the lemma.
\end{proof}

\medskip

\subsection{Ergodic optimization for generic Lorenz attractors}

In this subsection we will prove Corollary~\ref{cor:generic}.
Let $M$ be a 3-dimensional compact boundless Riemannian manifold and let
$\mathcal R \subset \mathfrak{X}^1(M)$ be the $C^1$-residual subset given by Proposition~\ref{MP}.

Given a vector field $X\in \mathcal R$, by item (1) of Theorem \ref{TeoALorenz} there exists a $C^0$-residual subset of observables $\mathcal R_X\subset C^0(N,\mathbb R)$
such that every $\Phi\in \mathcal R_X$ has a unique $\Phi$-maximizing measure $\mu$ with respect to $(X^t)_t$, whose support
contains a singularity and is not atomic.
Then, for $\mu$-almost every $x$ there exists
$y\in W^u(\sigma)\setminus\{\sigma\}$ such that $y\in \omega(x)$. Using continuous dependence on initial conditions and the fact that
$\overline{W^{u,\pm}(\sigma)}=\Lambda$ we conclude that $\supp \mu=\Lambda$.
By construction the residual subset
$\bigcup_{X\in \mathcal R} \, \{X\} \times \mathcal R_X \subset\mathfrak{X}^1(M) \times C^0(N,\mathbb R)$
satisfies the requirements of the corollary. \hfill $\square$

\subsection{Ergodic optimization for H\"older continuous observables }

Here we prove item (2) in Theorem \ref{TeoALorenz}.
Using that $C^\alpha(N,\mathbb R) \subset C^0(N,\mathbb R)$ is a $C^0$-dense subspace, Theorem~\ref{Jenkinson} and Remarks~\ref{rmk:Holder} and ~\ref{rmk:thm-unique-flows}, we conclude that there
exists a $C^\alpha$-residual subset $\mathfrak R_2 \subset C^\alpha(N,\mathbb R)$ so that every $\Phi\in \mathfrak R_2$ has a unique
$\Phi$-maximizing measure.
Given $\Phi\in \mathfrak R_2$, if $\mu$ denotes the unique $\Phi$-maximizing measure then there are three cases to consider:

\medskip
(i) If $M(\Phi)=\Phi(\sigma)$
then $\mu=\delta_\sigma$ is the Dirac measure at the singularity;

\medskip
(ii) If there exists $\vep_0>0$ such that $M_{\cU_{\vep}}(\Phi) = M(\Phi)$ for every $0<\vep\le \vep_0$, since the restriction
map $C^\al(N,\mathbb R) \to C^\al(\Lambda_{\cU_{\vep}},\mathbb R)$ is a submersion and generic H\"older observables on hyperbolic sets have maximizing measures supported on periodic orbits we conclude that the unique maximizing measure $\mu$ of a $C^\al$-Baire
generic observable
it is supported on a periodic orbit in $\Lambda_{\cU_{\vep_0}}$; and

\medskip
(iii) If $M(\Phi)> \Phi(\sigma)$
and $M_{\cU_{\vep}}(\Phi) < M(\Phi)$ for every $\vep>0$, the same argument involved in the proof of item (1) in Theorem \ref{TeoALorenz}
ensures that $\mu$
is not atomic and its support contains the singularity.

\medskip
This finishes the proof of item (2) in Theorem \ref{TeoALorenz}. \hfill $\square$

\begin{remark}
We expect that $C^\al$-generic observables have unique maximizing measures
supported at some critical element, similarly to the case of hyperbolic basic sets.
This is the case if one assures that item (iii) above holds for observables in a meager
subset. This is not immediate as the periodic measures for the restricted maximum $M_{\cU_{\vep}}(\Phi)$
could have smaller frequency of visits to neighborhoods of the singularity as $\vep$ tends to zero,
but still to accumulate on a invariant measures having the singularity in its support.
\end{remark}

\begin{remark}
The case when the unique maximizing measure is supported at a singularity is somewhat rare. In fact, Lemma~\ref{le:nonsingularm}
ensures that any $\Phi \in C^\al(N,\mathbb R)$ which has a maximizing measure supported at a singularity is $C^0$-approximated by $C^0$-open sets of observables in $C^\al(N,\mathbb R)$ for which Dirac measures at singularities are not maximizing measures.
In particular there exists a $C^\al$-open and $C^0$-dense subset of the residual subset $\mathcal R \subset C^\al(N,\mathbb R)$ formed by observables admitting a unique and periodic (non-singular)  maximizing measure.
\end{remark}

\section{Final comments}\label{sec:final}

The ergodic optimization for hyperbolic and singular-hyperbolic flows is still giving first steps and have very few contributions.
Let us describe some of the possible  future directions of research in this topic.
First, while one expects maximizing measures for typical
continuous observables to be fully supported and of zero topological entropy, we could not prove this in full strength in
Theorem~\ref{TeoALorenz}, as a limitation of the approximation method. Indeed, as maximizing measures are
obtained as weak$^*$ limits of maximizing measures with zero entropy and fully supported on the approximating horseshoes, these
may
have positive entropy. A first interesting question is whether the measure theoretic entropies associated
to these sequences of maximizing measures has some semi-continuity.  Since singular-hyperbolic attractors fail to satisfy the specification property \cite{SVY},
one cannot expect to use the methods in \cite{morris2010ergodic}. A second question that arises naturally is related to
the proof of Corollary~\ref{cor:generic} and the intrinsic nature of singular-hyperbolic attractors:
does every compact, nontrivial and transitive subset of a singular-hyperbolic attractor coincide with the attractor itself?

An alternative method to describe the ergodic optimization for three-dimensional singular-hyperbolic attractors
(including the Lorenz attractor) could follow from the construction of calibrated sub-actions, following the approach
of Lopes and Thieullen \cite{lopes2005sub} in the case of Anosov flows. Indeed, since singular-hyperbolic attractors
admit suitable cross-sections and can be modeled by suspension flows over maps with countable Markov partitions
(see e.g. \cite{araujo2010three}), it sounds reasonable
that the construction of sub-actions for Young towers in \cite{Branton} may
be pushed to the context of suspension flows of maps modeled by Young towers and ultimately to the realm of singular-hyperbolic attractors. The ergodic optimization for Lorenz attractors and H\"older potentials seems not immediate from this method. Indeed, the $C^0$-topology is crucial in order to perturb observables in a small open neighborhood of the singularities
(recall e.g. Lemma~\ref{le:nonsingularm}).

A further interesting question, inspired by \cite{contreras2001lyapunov},
concerns the case of  maximizing measures for the  Lyapunov exponents. In the case of three-dimensional
hyperbolic flows, this corresponds to the description of the maximizing measures with respect to the observable
$-\log |\det DX^t \mid E^u|$. In this case the vector field and the observable are coupled, which demands
subtle perturbations of the underlying dynamics.

Finally, it seems challenging to describe the ergodic optimization of continuous flows (or Lipschitz vector fields) when
the observable is fixed. Such a problem was addressed in \cite{AT} in the discrete-time context, but an extension to the
context of continuous flows seems to face fundamental difficulties raised by the lack of perturbation methods for
flows with low regularity.

\vspace{.1cm}
\subsection*{Acknowledgements}
This work was part of the first author's PhD thesis at UFBA, and it was partially supported by CNPq-Brazil and CAPES-Brazil.
The third author was partially supported by CNPq-Brazil and by FCT-Portugal.
The authors are indebted to V. Ara\'ujo, F. Pedreira, V. Pinheiro and X. Tian
for useful comments and calling their attention to \cite{HW,MP, Pedreira}.


\end{document}